\documentclass[11pt]{article}
\usepackage[utf8]{inputenc}
\usepackage[includeheadfoot,a4paper, centering,scale=0.8,headheight=15pt]{geometry}
\usepackage{fancyhdr}
\usepackage{lipsum}
\usepackage{lastpage}
\usepackage[dvipsnames,svgnames]{xcolor}
\usepackage[strict]{changepage} % 
\usepackage{framed} % 
\usepackage{graphicx}
\usepackage{amssymb}
\usepackage{latexsym}
\usepackage{mathtools}
\usepackage{multirow}
\usepackage{textcomp} 
\usepackage{amsthm}
\usepackage[all]{xy}
\usepackage{color}

\usepackage{blindtext}
\usepackage{textcomp}
\usepackage[alphabetic]{amsrefs}
\usepackage[colorlinks,linkcolor=red,anchorcolor=yellow,citecolor=blue]{hyperref}

\usepackage[T1]{fontenc}
\usepackage{lmodern}
\usepackage[theoremfont]{newpxtext}
\usepackage[type1]{sourcesanspro}
\usepackage[varg]{newpxmath}
\usepackage{bm}
\usepackage[scr=boondox,  % heavily sloped
            cal=esstix]   % slightly sloped
           {mathalpha}
\usepackage{verbatim}

\usepackage{titlesec}
\titleformat{\section}
  {\normalfont\fontsize{14}{16}\bfseries\scshape}{\thesection}{1em}{}

\newtheoremstyle{break}
{\baselineskip} % space above %.5\baselineskip\@plus.2\baselineskip\@minus.2\baselineskip
{\baselineskip}% space below
  {} %Body font
  {}%Indent amount (empty = no indent, \parindent = para indent)
  {\bfseries}%Thm head font
  {}%Punctuation after thm head
  {\newline}% Space after thm head: " " = normal interword space;
        {}% Thm head spec
\theoremstyle{break}
\newtheorem{theo}{Theorem}[section]

\numberwithin{equation}{section}
\newtheorem{theorem}[theo]{Theorem}
\newtheorem{definition}[theo]{Definition}
\newtheorem{lemma}[theo]{Lemma}
\newtheorem{prop}[theo]{Proposition}
\newtheorem*{prop*}{Proposition}
\newtheorem{example}[theo]{Example}
\newtheorem{corol}[theo]{Corollary}

\newtheorem{remark}[theo]{Remark}

\newcommand{\bijar}[1][]{%
    \ar[#1]
    \ar@<0.7ex>@{}[#1]|-*[@]{\sim}} 
\newcommand{\locr}[1]{\mathcal{O}_{ #1 }}
\newcommand{\tx}[1]{\text{ #1}}

\newcommand{\dlog}[1]{\text{dlog}(#1)}

\newcommand{\len}[1]{\text{len}( #1)}

\DeclareMathOperator{\lcm}{lcm}

\title{\bfseries \scshape Brauer $p$-dimension and Kato's Swan Conductor}
\author{Yizhen Zhao}
\date{\today}

\begin{document}

\maketitle

\begin{abstract}
We use Kato's Swan conductor to study the Brauer $p$-dimension of fields of characteristic $p>0$. We mainly investigate two types of fields: henselian discretely valued fields and semi-global fields. While investigating the Brauer $p$-dimension of semi-global fields, we use a Gersten-type sequence to analyse the ramification behavior of a Brauer class in a $2$-dimensional regular local ring. Using this result, we give a partial result on the Brauer $p$-dimension of function fields of algebraic curves over $\bar{k}((t))$ with good reduction.
\end{abstract}
\tableofcontents
\section{Introduction}

Let $k$ be a field. For any $k$-central simple algebra $A$, we denote by $\textup{per}(A)$ the order of its class in the Brauer group $\textup{Br}(k)$ (called the period) and by $\textup{ind}(A)$ its index which is the gcd of all the degrees of the finite splitting fields. It is well-known that $\textup{per}(A)$ divides $\textup{ind}(A)$, and these two integers have the same prime factors. So the period is bounded by the index and the index is bounded above by a power of the period. We use notion of the Brauer dimension to make this relationship precise. For $n\in \mathbb{N}$, denote by $A[n]$ the $n$-torsion part of the group $A$. Then, for a prime $p$, define the Brauer dimension at $p$ as follows
 \[
\textup{Br.dim}_p(k):=\min\limits_d\left\{
     \begin{array}{cl}
  \text{ind}(A)\mid \text{per}(A)^d & {\textup{for any}\  A\in \text{Br}(k)[p^n] \ \text{and} \ n\in \mathbb{N};} \\
  \infty & {\textup{otherwise}.}
 \end{array}
 \right.
    \]
Then define the Brauer dimension of $k$ to be 
\[\textup{Br.dim}(k)=\sup_p\left\{\textup{Br.dim}_p(k)\right\}.
\]
The period-index problem of a field $k$ is to investigate the Brauer dimension of $k$.

One important class of fields for the period-index problem is that of the $C_m$ fields. The class of $C_m$ fields includes the class of function fields of dimension $m$ algebraic varieties over an algebraic closed field. Michael Artin \cite{MR0657430} conjectured that $\text{Br.dim}(k)\leq 1$ for a $C_2$ field $k$. This conjecture has been proved in many cases by several authors. This conjecture has a natural extension to all $C_m$ fields: Is it true that $\text{Br.dim}(k)\leq m - 1$ for any $C_m$ field $k$? Many recent works on the period-index problem have been motivated by this question. The progress on the period-index problems was discussed in \cite{MR3413868}. 

We next discuss our results. Before stating our theorems, let us recall some known results. 

\begin{enumerate}
\item In \cite{MR2545681}, Harbater, Hartmann and Krashen prove the following: If $F$ is a complete discretely valued field with residue field $k$ such that $\text{Br.dim}_p(l)\leq d$ for all finite extension $l/k$ and all primes $p\neq \text{char}(k)$, then $\text{Br.dim}_p(F)\leq d+1$ for all primes $p\neq \text{char}(k)$.
\item In \cite{MR1462850}, Saltman proves the following: For a finitely generated field $F$ of transcendence degree $1$ over an $l$-adic field, $\text{Br.dim}_p(F)=2$ for every prime $p\neq l$. 
\end{enumerate}

In this paper, we focus on the case which is {\it orthogonal} to these previous cases: We consider the Brauer dimension at $p$ for fields of characteristic $p$. In this case, we prove two theorems corresponding to the two results above.

The first result is the following:

\begin{theorem}
\label{firstthe}
Let $F$ be a henselian discretely valued field of characteristic $p>0$ with the residue field $k$ such that $[k:k^p]=p$ and $\text{Br.dim}_p(l)=0$ for any finite extension $l/k$. Then we have $\text{Br.dim}_p(F)=1$.
\end{theorem}
This result was first proved by Kato \cite{MR550688} for the $p$-torsion classes in the case of complete discretely valued fields. The result can be generalized to the henselian discretely valued case directly. For the proof of our next result, we need this extension to the henselian case. 

The key ingredient in the proof of Theorem \ref{firstthe} is Kato's Swan conductor (Definition \ref{defsw}). Kato's Swan condutor is defined for an element in the $p$-torsion part of the Brauer group of a field $k$ over characteristic $p$. It measures the {\it wildness} of a Brauer class and this phenomenon exists only over characteristic $p$. We also define an extended version of the Swan conductor that depends on a choice of a model (say $\mathcal{X}$) and call it the $\mathcal{X}$-Swan conductor in our second result based on Kato's Swan conductor for a discretely valued field.

We discuss the definition of the $\mathcal{X}$-Swan conductor now. Let $X$ be a nonsingular projective curve over ${k}((t))$, where ${k}$ is an algebraically closed field of characteristic $p$. Suppose there exists a model  $\mathcal{X}$ of $X$ over $\text{Spec}\ k[[t]]$ with good reduction. The purity theorem for Brauer groups implies that the Brauer group of $X$ is a subgroup of $\text{Br}~k(X)\cong \text{Br}~k(\mathcal{X})$. Let $\omega\in \textup{Br}(X)[p]$. Again by the purity of Brauer groups, $\omega$ ramifies only along the closed fiber $Y$ of $\mathcal{X}$ over $(t)$, where $Y\cong\mathcal{X}\times \text{Spec}~ k$. Then we define the  $\mathcal{X}$-Swan conductor of $\omega$ to be Kato's Swan conductor of $\omega$ with respect to the valuation given by $Y$. In this article, we use the $\mathcal{X}$-Swan conductor as a technical tool. We denote it by $\text{sw}_\mathcal{X}(\omega)$. With this definition of the $\mathcal{X}$-Swan conductor, we have the following second theorem.
\begin{theorem}[Theorem \ref{swp}] \label{geoperind} 
Let $X$ be a smooth projective curve over $k((t))$ where $k$ is an algebraically closed fields of characteristic $p$. Suppose there is a model $\mathcal{X}$ over $k[[t]]$ with good reduction. Suppose that $\omega\in \textup{Br}(X)[p]$ satisfies $\textup{sw}_{\mathcal{X}}(\omega)< p$. Then $\textup{per}(\omega)=\textup{ind}(\omega)$.    
\end{theorem}
The function field $k(X)$ of $X$ is  finitely generated and of transcendence degree $1$ over the $C_1$ field $k((t))$. Hence it is a $C_2$ field and we expect that $\text{Br.dim}_p\ k(X)=1$. Our second result can be viewed as progress towards the computation of Brauer dimension of $k(X)$ at $p$.

To prove Theorem \ref{geoperind}, we use the patching method \cite{MR2545681} and follow the strategy as in \cite{MR3219517}. It reduces the global period-index problem to the period-index problem along the completion of the local ring at the generic point of the closed fiber and of local rings at closed points in the closed fiber (Theorem \ref{patching}). For the completion at the generic point, the period-index result is provided by Theorem \ref{firstthe}. For the completion at the closed points, we use a Gersten-type exact sequence (Theorem \ref{Gersten}) to analyse the ramification of a Brauer class locally. This Gersten-type exact sequence is the replacement of the Artin-Mumford ramification sequence \cite{MR321934}. Recall that the Arin-Mumford sequence provides information about ramification when the torsion of the Brauer classes is not divisible by the base characteristic. So, it does not apply in our case. We show that the $\mathcal{X}$-Swan conductor controls the local ramification behavior (Proposition \ref{localest}).

Finally, we remark that the condition on the $\mathcal{X}$-Swan conductor is used in avoiding the difficulty of the problem when the Swan conductor is greater than $p$. Our work is still in progress for the general period-index problem of the function field $k(X)$.
  \subsection*{Outline}
  The paper is organized as follows. In Section 2, we introduce some notation and recall some basic definitions. In Section 3, we review Brauer groups and concentrate on the $p$-torsion in Brauer groups in characteristic $p$. For the $p$-torsion in Brauer groups, it is closely related to the logarithmic differential forms. We summarize the results in Section \ref{logintro}. In Section $4$, we introduce Kato's Swan conductor of a Brauer class and generalize Kato's proof to the henselian case. In Section $5$, we prove Theorem \ref{firstthe} and recover Chipchakov's result \cite{MR4453883} with Kato's Swan conductor. In Section $6$, we discuss the ramification of central division algebras over the field of Laurent series as applications. Finally, in Section 7, we define the $\mathcal{X}$-Swan conductor and prove Theorem \ref{geoperind}. We compute the Gersten-type exact sequence of logarithmic differential forms for the complete local ring $\bar{k}[[\pi,t]]$ where $\text{char}(k)=p>0$.
  \subsection*{Acknowledgments}
I would like to express my great gratitude to my advisor, Professor Rajesh Kulkarni, who advised me with this project in Brauer groups. Moreover, I am really grateful for the discussion with Professor Ivan Chipchakov and Professor Takao Yamazaki. During the preparation of the project, I was partially supported by the NSF grant DMS-2101761.
%%-----------------------------------------------------
\section{Notations}
Let $k$ be a field. Denote by $\bar{k}$ the algebraic closure of $k$.
Let $X$ be a scheme and $A$ be an \'etale sheaf over $X$.
Denote by $H^n(X,A)$ the $n$-th cohomology group of the small \'etale site of $X$ with values in $A$.
When $X=\text{Spec}\ k$, we also use $H^n(k,A)$ for $H^n(X,A)$.

Let $R$ be a ring and $\mathfrak{p}$ be a prime ideal of $R$. We denote by $R^h_\mathfrak{p}$ the henselization of the localization $R_p$ of $R$ at the prime ideal $\mathfrak{p}$.
\subsection{de Rham-Witt Complex}
\label{rederham}
Let $X$ be a regular scheme of dimension $d$ over a perfect field of characteristic $p>0$, and let $W_n(k)$ be the ring of Witt vectors of length $n$. Recall that the de Rham-Witt complex $W\Omega^\bullet_{X/k}$ is the inverse limit of an inverse system $(W_n\Omega^\bullet_{X/k})_{n\geq 1}$ of complexes where
\[
W_n\Omega^\bullet_{X/k}:=(\xymatrix{
W_n\Omega^0_{X/k}\ar[r]^-d& W_n\Omega^1_{X/k}\ar[r]& \cdots\ar[r]^-d&W_n\Omega^i_{X/k}\ar[r]^-d&\cdots
}
)
\]
of sheaves of $W_n\mathcal{O}_X$-modules on the Zariski site of $X$. Moreover, it is already a complex of sheaves on the small \'etale site of $X$ (\cite{MR565469}, also Proposition \ref{etalesheafomega}). The complex $W_n\Omega^\bullet_{X/k}$ is called the de Rham-Witt complex of level $n$.

We have the following operators on the de Rham-Witt complex:
\begin{enumerate}
    \item The projection $R:W_n\Omega^\bullet_X\rightarrow W_{n-1}\Omega^\bullet_X$, which is a surjective homomorphism of differential graded algebras.
\item The Verschieburg $V:W_n\Omega^\bullet_X\rightarrow W_{n+1}\Omega^\bullet_X$, which is an additive homomorphism.
\item The Frobenius $F:W_n\Omega^\bullet_X\rightarrow W_{n-1}\Omega^\bullet_X$, which is a homomorphism of differential graded algebras.
\end{enumerate}

The logarithmic de Rham-Witt sheaf $W_n\Omega^i_{X,\textup{log}}$ \cite{MR2396000} is defined by
\[
W_n\Omega^i_{X,\textup{log}}:=\textup{Im}(s:(\mathcal{O}^\times_X)^{\otimes i}\rightarrow W_n\Omega^i_X),
\]
where $s$ is defined by
\[
s(x_1\otimes\cdots\otimes x_i):=\dlog{x_1}\wedge\cdots\dlog{x_i}.
\]
(Here $\textup{Im}$ is considered in the category of sheaves on the small \'etale site of $X$.)
\subsection{\texorpdfstring{$\bm{C_m}$ fields}{Lg}}
We recall the definition of $C_m$ fields as these are important for considerations in this article. For any positive integer $m$, we say a field $k$ satisfies condition $C_m$ if every homogeneous polynomial $f\in k[x_1,\cdots,x_n]$ of degree $d$ with $d^m<n$ has a nontrivial zero in $k^n$ \cite{MR0046388}.

Here are some properties of $C_m$ fields:
\begin{enumerate}
    \item If a field is $C_m$, then any finite extension is also $C_m$.
    \item If a field is $C_m$, then any extension of transcendence degree $n$ is $C_{m+n}$.
    \item If a field $k$ is $C_m$, then $k((t))$ the field of Laurent series is $C_{m+1}$.
    \item The Brauer group of a $C_1$ field is $0$ \cite{MR554237}.
\end{enumerate}

Examples of $C_m$ fields:
\begin{description}
\item[\bm{$C_0$} fields] These are precisely the algebraically closed fields.
\item[\bm{$C_1$} fields (quasi-algebraically closed fields)] 
$ $
\begin{itemize}
    \item Finite fields
    \item The maximal unramified extension of a complete discretely valued field with a perfect residue field
    \item Complete discretely valued fields with algebraically closed residue fields.
\end{itemize} 
\item[\bm{$C_m$} fields] If $V$ is a variety of dimension $m$ over an algebraically closed field $k$, then the function field $k(V)$ is $C_{m}$.
\end{description}
\subsection{\'Etale Motivic Cohomology}
\label{etalemotivic}
Let $k$ be a perfect field of characteristic $p>0$ and $X$ be a smooth scheme over $k$. Then Voevodsky's object $\mathbb{Z}/p^r(j)$ in the derived category of Zariski (or \'etale) sheaves on $X$ is isomorphic to $W_r\Omega^j_{\textup{log}}[-j]$ \cite{MR1738056}*{Proposition 3.1, Theorem 8.3}. By \'etale motivic cohomology, we mean the \'etale cohomology of $X$ with coefficients in $\mathbb{Z}/p^r(j)$. It can be written in terms of differential forms as:
\[
H^i(X,\mathbb{Z}/p^r(j))\cong H^{i-j}(X,W_r\Omega^j_{X,\textup{log}}).
\]
This also applies to the case when $X = \textup{Spec}(k)$  for a field $k$ of characteristic $p$, not necessarily perfect. Since the \'etale $p$-cohomological dimension of the field $k$ is at most $1$ \cite{MR3727161}*{Proposition 6.1.9}, $H^i(k,\mathbb{Z}/p^r(j))$ is zero except when $i$ is $j$ or $j+1$. When $i=j$, Bloch and Kato identified this group with the Milnor $K$-group $K^M_j(k)/p^r$, or also with the group $W_r\Omega^j_{\textup{log},k}$ \cite{MR0849653}*{Corollary 2.8}.
\section{Backgrounds on Brauer groups}
\subsection{Basic properties of Brauer groups}
A central simple algebra (CSA) over a field $K$ is a finite-dimensional associative $K$-algebra $A$ that is simple with center $K$.
  
Two central simple algebras $A, A'$ are called {\it Morita equivalent} if there exist integers $r, s\in\mathbb{N}$ such that $A\otimes M_r(K)\simeq A'\otimes M_s(K)$ as $K$-algebras. By the Artin-Wedderburn theorem, a finite-dimensional simple algebra $A$ is isomorphic to the matrix algebra $M_n(D)$ for a $K$-central division algebra $D$. Moreover, such a division algebra is uniquely determined by a central simple algebra.

The {\it Brauer group} of a field $K$ is a torsion abelian group whose element are Morita equivalence classes of central simple algebras over $K$. The addition in the Brauer group is given by the tensor product of algebras.

As mentioned above, there is a unique division algebra in each Brauer  class. The {\it degree} $\text{deg}(A)$ of a central simple algebra $A$ is the integer $n$ such that $\text{dim}_K(A)=n^2$. Then we define the {\it index} $\text{ind}(A)$ of a central simple algebra $A$ to be the degree of the division algebra $D$ associated to $A$ by the Artin-Wedderburn theorem. In particular, note that the index is well-defined for a Brauer class. Also, for a Brauer class $[A]$ associated to a central simple algebra $A$, the {\it period} $\text{per}(A)$ is its order in the Brauer group $\text{Br}(K)$. 

It is well-known that the period divides the index of a central simple algebra, and these two integers have the same prime factors. So the index divides a power of the period. The period-index problem asks if one can bound the index in terms of the power of the period. Here are the relevant definitions from the introduction.

\begin{definition}[Brauer dimension, \cite{MR3413868}]
\label{defofbrdim}
\leavevmode \vspace{-\baselineskip}
\begin{itemize}
\item
  Let $K$ be a field. For a prime $p$, {\it the Brauer dimension at $p$}, $\textup{Br.~dim}_p(K)$, is the smallest integer $d$ such that for any $A\in \text{Br}_{p^n}(L)$, $\text{ind}(A)|\ \text{per}(A)^d$, and $\infty$ if no such number exists. 
 \item {\it The Brauer dimension of $K$} is 
\[\textup{Br.dim}(K)=\sup_p\left\{\textup{Br.dim}_p(K)\right\}.
\]
\end{itemize}
\end{definition} 
The period-index problem asks if $\text{Br.dim}(K)$ is finite, and the local period-index problem asks if $\text{Br.dim}_p(K)$ is finite for an arbitrary prime $p$. 

\iffalse When $\text{char}(K)=p>0$, combining the result on generic matrices \cite{MR2785502} and Florence's result \cite{MR3103068} shows that $\text{Br.dim}_p(F)$ is finite. Hence, the local period-index problem at $p$ is reduced to find a bound or the precise value.\fi

The Brauer group can also be defined in terms of Galois cohomology. We have 
\[\tx{Br}(K)=H^2(K,\mathbb{G}_m),\]
where $\mathbb{G}_m$ denotes the sheaf of units in the structure sheaf.

In general, the Brauer group of a scheme is defined in terms of Azumaya algebras. An Azumaya algebra is a generalization of central simple algebras to $R$-algebras where $R$ may not be a field.  For a scheme $X$, an Azumaya algebra on $X$ with structure sheaf $\mathcal{O}_X$ is a coherent sheaf $\mathcal{A}$ of $\mathcal{O}_X$-algebras that is \'etale locally isomorphic to the sheaf of matrices over the structure sheaf. The Brauer group $\tx{Br}(X)$ is an abelian group of equivalence classes of Azumaya algebras, with the addition given by the tensor product of algebras. Here two Azumaya algebras $\mathcal{A}, \mathcal{A'}$ are considered to be equivalent when $\textup{M}_r(\mathcal{A}) \cong \textup{M}_s(\mathcal{A'})$ as sheaves of $\mathcal{O}_X$-algebras for matrices of size $r \times r$ and $s \times s$ resp..

As in the case of a field, we define the cohomological Brauer group of a quasi-compact scheme $X$ to be the torsion subgroup of the \'etale cohomology group $H^2(X,\mathbb{G}_m)$. The cohomology group $H^2(X,\mathbb{G}_m)$ is torsion for a regular scheme $X$, but it may not be torsion in general. 

We recall several well-known facts about Brauer groups in the following.

\begin{theorem}
The Brauer group of a scheme $X$ is equal to the cohomological Brauer group for any scheme with an ample line bundle.
\end{theorem}
For example, when $X$ is quasi-projective over a field $k$, we have the coincidence of two Brauer groups.

\begin{theorem}[Purity in codimension $1$, \cite{MR3959863}]
\label{purity}
For a Noetherian, integral, regular scheme $X$ with function field $K$,
        \[
       H^2(X,\mathbb{G}_m)=\bigcap_{x\in X^1} H^2(\mathcal{O}_{X,x},\mathbb{G}_m) \ \text{in} \ H^2(K, \mathbb{G}_m).\]
\end{theorem} 

\subsection{Structure of \texorpdfstring{$p$}{Lg}-primary part of Brauer groups}
\label{logintro}
In this subsection, we assume all the fields have positive characteristic $p>0$.
We focus on the $p$-primary part of the Brauer groups. First, we recall the $p$-primary counterpart of the Merkurjev-Suslin theorem \cite{MR675529}. The Merkurjev-Suslin theorem states that $\text{Br}_n(K)$ is generated by cyclic algebras of degree $n$ when $K$ contains a primitive $n$-root of unity $\mu_n$.

Firstly, we recall the Artin-Schreier-Witt theory of cyclic field extensions in positive characteristic:

\begin{theorem}[\cite{MR554237}]
    Let $k$ be a field of characteristic $p>0$. Denote by $\mathcal{P}:W_r(k)\rightarrow W_r(k)$ the endmorphism of the length-$r$ Witt ring that maps $(x_1,\cdots,x_r)\in W_r(k)$ to $(x_1^p,\cdots,x_r^p)-(x_1,\cdots,x_r)$. Then there exists a canonical isomorphism
    \[
W_r(k)/\mathcal{P}(W_r(k))\cong H^1(k,\mathbb{Z}/p^r).
    \]
\end{theorem}
Then we have the following theorem about the $p^r$-cyclic algebras (symbol algebras).
\begin{prop}
For every $\omega\in \tx{Br}(K)[p^r]$, we can write
\[\omega=\sum_i [a_i,b_i),
\]
as a sum of $p^r$-symbol algebras where $a_i\in W_r(K)$ and $b_i \in K^\times$. The $p^r$-symbol algebra $[a_i, b_i)$ is defined by
\[ [a_i,b_i):=\left\langle x,y \left\vert
\begin{aligned}
&{\textup{ $x$ is a primitive element of the Artin-Schreier-Witt extension defined by}}\\
&{\textup{ 
$\mathcal{P}(x_1,\dots,x_r)=a_i$ and with a generator $\sigma$ of the Galois group such that,
}
}\\
& \begin{array}{lr}
y^{p^n} =  b_i, & y^{-1}x y =  \sigma(x).
\end{array}
\end{aligned}
\right.
\right\rangle.
\]
\end{prop}
The $p^r$-symbol algebra has index = period =  $p^r$.

\begin{example}
Let $a\in K,~ b\in K^\times$ and consider the $p$-symbol algebra $[a,b)$. By definition,
\[ [a,b):=\langle x,y\mid x^p-x=a, y^{p}=b, y^{-1}xy=x+1\rangle.
\]
This symbol algebra is the main object of our study since we can reduce questions related to $p^r$-torsion Brauer classes to the $p$-symbol algebra by Theorem \ref{hightop}.
\end{example}

Next we relate the $p$-primary part of the Brauer group with the de Rham-Witt complex $W_r\Omega_K^1$ \cite{MR565469} (Section \ref{rederham}). We can identify $\text{Br}(K)[p^r]$ with the cokernel of 
\begin{equation}
    \label{311}
F-I:W_r\Omega_K^1\rightarrow W_r\Omega_K^1/dV^{r-1} (K).    
\end{equation}
\begin{lemma}[\cite{MR565469}]
Let $K$ be a field of characteristic $p>0$ and $r\in \mathbb{N}^+$.
    \begin{align}
        \text{Br}(K)[p^r]\cong W_r\Omega^1_K/\left((F-I)W_r\Omega^1_K+dW_r(K)\right).
    \end{align}
\end{lemma}
\begin{proof}
$ $\\
In fact, there exists an exact sequence of \'etale sheaves over the affine scheme $X=\text{Spec}(K)$:
\[  
  \xymatrix{
  0\ar[r] & W_r\Omega_{X,\text{log}}^1  \ar[r] & W_r\Omega^1_X \ar[r]^-{F-I} & W_r\Omega^1_X/dV^{r-1}(\mathcal{O}_X) \ar[r] & 0
  },
  \]
which induces the cohomology sequence
  \[
  \xymatrix{
  W_r\Omega_K^1 \ar[r]^-{F-I} & W_r\Omega_K^1/dV^{r-1}K \ar[r]^-{\delta_r} & H^1(K,W_r\Omega_{K,\text{log}}^1)\ar[r] & 0 \ ,
  }
  \]
  since $H^1(K,W_r\Omega^1_K)=0$ by the quasi-coherence of $W_r\Omega^1_K$.  
  
  Also, there is another exact sequence of \'etale sheaves which relates the $p^r$-torsion part of the Brauer group with the logarithmic de Rham-Witt complex
  \[
  \xymatrix{
  0\ar[r]& \mathbb{G}_m\ar[r]^-{p^r} & \mathbb{G}_m\ar[r] & W_r\Omega_{\text{log}}^1\ar[r]& 0.
  }
  \]
  It induces the long exact cohomology sequence
  \[
  \xymatrix{
 0\ar[r]& H^1(K,W_r\Omega_{K,\text{log}}^1)\ar[r] & H^2(K,\mathbb{G}_m)\ar[r]^-{p^r} & H^2(K,\mathbb{G}_m),
  }
  \]
  where $H^1(K,\mathbb{G}_m)=0$ by Hilbert's Theorem 90.
  
  By using the relation $d=F^{r-1}dV^{r-1}$ (\cite{MR565469} 2.18), it is easy to see that $\delta_r$ induces an isomorphism between the cokernel of (\ref{311}) and $\textup{Br}(K)[p^r]$.
  \end{proof}
  Now we use the structure of $W_r\Omega^1_K$ to describe $p^r$-torsion part of the Brauer group of $K$.
We recall some facts about $W_r\Omega^1_K$ \cite{MR3818286}. We use the notation $[a]_r:=(a,0,\cdots,0)\in W_r(K)$.

\begin{definition}
$M^1_r K \subset W_r\Omega_K^1$ denotes the subgroup generated by the elements $[a]_rd[f]_r$ where $a\in K,f\in K^\times$.
\end{definition}

\begin{lemma}[\cite{MR3818286}*{Lemma 2.4}]
Let $M^1_r K\subset W_r\Omega^1_K$ denote the subgroup generated by {\it multiplicative} elements $[a]_rd[f]_r$. Then we have
\[W_r\Omega^1_K=\sum_{i=0}^{r-1}V^i M^1_{r-i} K +\sum_{i=0}^{r-1}d V^i K.\]
Moreover,
\[dW_r(K)=\sum_{i=0}^{r-1} d V^i K\subset W_r\Omega^1_K.\]
\end{lemma}

It follows that
\begin{align}
    \text{Br}_{p^r}(K)\simeq W_r\Omega^1_K\left/ \left((F-I)W_r\Omega^1_K+dW_r(K)\right)\simeq \sum_{i=0}^{r-1}\left[V^i M^1_{r-i} K\right]. \right. \label{br}
\end{align} 
Then we relate the differential forms with symbol algebras by the following map
\begin{align*}
    \delta_r:\ W_r\Omega^1_K/((F-I)W_r\Omega^1_K+dW_r(K))& \longrightarrow \text{Br}(K)[{p^r}]\\
    a\ \dlog{[b]_r} & \longmapsto [a,b),
\end{align*}
where $a\in W_r(K),\ b\in K^\times$, and $\dlog{[b]_r}=[b]_r^{-1}d[b]_r$.

We denote the composite map $W_r\Omega_K^1\rightarrow \text{Br}(K)[{p^r}]\rightarrow \text{Br}(K)$ by $\delta_r$ as well. We have a commutative diagram 
\[
\xymatrix{
& W_{r-1}\Omega_K^1  \ar[d]_{\delta_{r-1}}  \ar[r]^V &  W_r \Omega_K^1 \ar[ld]_{\delta_r}\\
& \text{Br}(K)
}
\]

Using the isomorphism (\ref{br}), it is easy to give a direct proof of the following theorem.
\begin{theorem}[\cite{MR689394}]
\label{hightop}
For a field $K$ of positive characteristic $p>0$ and $m\in \mathbb{N}$, we have an exact sequence:
\begin{align}
    \xymatrix{
0\ar[r] & \text{Br}(K)[{p^m}]\ar[r]^V & \text{Br}(K)[{p^{m+1}}]\ar[r]^{R^1} & \text{Br}(K)[p]\ar[r] & 0,
}
\end{align}
where $R^1:W_m \Omega^1_K\rightarrow \Omega^1_K$ sends $[a]_m\dlog{[b]_m}$ to $a\ \dlog{b}$.
\end{theorem}

%-------------------------
\iffalse
\begin{corol}
For a field $K$ of positive characteristic $p>0$, 
\[\len{p^n,K}\leq n\cdot\len{p,K},\]
when $\len{p,K}<\infty.$
\end{corol}

This bound is not optimal in general. For example, Matzri\cite{MR3413868} shows that $\len{p^n,K}\leq m$ for a $C_m$ field with $\text{char}(K)=p>0$. 
\fi
%--------------------

The Brauer dimension at $p$ for a field of characteristic $p>0$ is effectively controlled by the \textbf{rank} of the $p$-basis.

\begin{definition}[$p$-basis and $p$-rank]
    Let $K$ be a field and $[K:K^p]=p^n$, $n\geq 0$. A $p$-basis of $K$ is a subset $\{x_i\}\subset K$ such that the elements $x^E=\prod x_i^{e_i},0\leq e_i<p$ form a basis of $K$ over $K^p$, and the $p$-rank of $K$ is the number of elements in the subset $\{x_i\}$. Hence the $p$-rank of $K$ is $n$.
\end{definition}

\begin{prop}[\cite{MR4176776}*{Corollary 3.4}]
Let $K$ be a field with $[K:K^p]=p^n$. Then $\text{Br.dim}_p(K)\leq n$.
\label{prank}
\end{prop}

\begin{proof}
For $r\in\mathbb{N}$ and a $p$-basis $\{a_i\}_{i=1}^n$ of $K$,  by Theorem \ref{hightop} and induction, every $p^r$-torsion Brauer class can be written as a sum of $n$ symbol algebras $[c_i,a_i)$, where $c_i \in W_r(K)$ for $i\in \{1,\cdots, n\}$. Then the proposition follows from the following lemma.
\end{proof}
\begin{lemma}[\cite{MR0000595}*{Ch. VII, Lemma 13}]
Let $K$ be a field of characteristic $p>0$.
If $A,B$ are two symbol algebras of degree $p^m$ and $p^n$ respectively, then $A\otimes B$ is Brauer equivalent to a symbol algebra of degree no more than $p^{m+n}$.
\end{lemma}

\subsection{Brauer group of a complete discretely valued field}\label{cdf}
In this subsection, $F$ denotes a complete discretely valued field with valuation ring $\mathcal{O}_F$, residue field $k$ and maximal ideal $m_F=(\pi)$. The valuation of $F$ is denoted by $v_F$. Recall that a discrete valuation is a map $v_F:\ F\rightarrow \mathbb{Z}\cup \{\infty\}$ that satisfies:\
\begin{enumerate}
    \item[(\romannumeral 1)] $v_F(a)=\infty$ if and only if $a=0$;
    \item[(\romannumeral 2)] $v_F(ab)=v_F(a)+v_F(b)$;\
    \item[(\romannumeral 3)] $v_F(a+b)\geq \min(v_F(a),v_F(b))$, with equality if $v_F(a)\neq v_F(b)$.\
\end{enumerate}
The valuation ring $\mathcal{O}_F = v_F^{-1}(\mathbb{Z}_{\geq 0})$ is a complete local ring.  For a complete discretely valued field $F$, we can extend the complete valuation $v_F$ to central simple divison algebras over $F$ and consider the residue division algebras. They are summarized in the following proposition.
\begin{prop}[\cite{StableOrders}*{Proposition 1.3.1}]
Let $D$ be a central division $F$-algebra.\
\begin{enumerate}
\item[(\romannumeral 1)] The function $w:D\rightarrow \mathbb{Z}\cup \{\infty\}$ defined by $w(a)=v_F(\text{det}(a))$ is a discrete valuation on $D$.\

\item[(\romannumeral 2)] The set $B:=\{a\in D\mid w(a)\geq 0\}=\{a\mid \text{det}(a)\in \locr{F}\}$ is the unique maximal $\locr{F}$-order in $D$.\

\item[(\romannumeral 3)] $B$ is a local domain with maximal ideal $J:=\{a\mid w(a)>0\}$; the residue ring $\Delta=B/J$ is a division ring.

\item[(\romannumeral 4)] If $\pi$ is an element of $J$ such that $w(\pi)$ takes the minimal positive value, then $J=B\pi=\pi B$
\end{enumerate}
\end{prop}

Next we study the unique maximal order $B$ in the above proposition. Let $k'$ be the center of $\Delta$. Then we have the integers $d,e,e',f,n$ defined as follows:
\begin{align}
    d=w(\pi),\ J^e=m_F B,\ e'=[k':k],\ f^2=[\Delta:k'],\ n^2=[D:F].\label{ram}
\end{align}
Here $n$ is the degree (index) of $D$, and also its degree.

\begin{lemma}[\cite{StableOrders}*{Lemma 1.3.7
}]
$ed=n,\ \text{and}\  ee' f^2=n^2.$ \label{ed}
\label{structureofrda}
\end{lemma}
\begin{corol}
    $[D:F]=1$ if and only if $[\Delta:k]=1$.
\label{artin nonzero}
\end{corol}
\begin{proof}
This is immediate from the above lemma.
\end{proof}
In this paper, we are interested in the case that the residue field $k$ is \emph{quasi-algebraically closed}, i.e a $C_1$ field. Recall that a finite extension of a $C_1$ field is also $C_1$. Hence, the central division algebra $\Delta$ over $k'$ will be isomorphic to $k'$, since $\text{Br}(k')=0$. This implies $f=1$. 
\begin{lemma}
Suppose that the residue field $k$ is $C_1$ and $[k:k^p]=p$. Then $e=e'=n$ and $d=1$.
\end{lemma}
\begin{proof}
    $ $\\
We already have $f=1$ and so it suffices to show $e'\leq n$ by Lemma \ref{structureofrda}. We will show that any field extension of $k$ is simple. In this case, $k'=k[\alpha]$ and we choose $\beta\in B$ such that $\bar{\beta}=\alpha\in k'$. Then we have $e'\leq [F(\beta):F]\leq n$, since $\text{ind}(D)$ is $n$.

Now we prove that any finite field extension $k'$ of $k$ is simple. The field extension $k\subset k'$ can be written as a chain of field extensions $k\subset l\subset k'$ such that $l$ is separable over $k$ and $k'$ is purely inseparable over $l$. It follows that $[l:l^p]=p$ by Lemma \ref{extofprank} below. Then the purely inseparable extension $k'/l$ is simple. Set $k'=l[\alpha_1]$ and $\alpha_1$ is algebraic over $k$. We can also denote $l=k[\alpha_2]$ by Theorem \ref{primitive} below, since $l/k$ is finite and separable. Finally, we get $k\subset k'=k[\alpha_1,\alpha_2]$ is simple by Theorem \ref{primitive} again.
\end{proof}
\begin{lemma}[\cite{MR1994218}*{A.V.135, Corollary 3}]
Let $l/k$ be a finite or separable field extension of fields of characteristic $p$, and let $n$ be the $p$-rank of $k$. Then the $p$-rank of $l$ is also $n$.
\label{extofprank}
\end{lemma}
\begin{theorem}[\cite{milne2022}*{Theorem 5.1}]
    Let $E=F[\alpha_1,\cdots,\alpha_r]$ be a finite extension of $F$, and assume that $\alpha_2,\cdots,\alpha_r$ are separable over $F$ (but not necessarily $\alpha_1)$. Then there exists a $\gamma\in E$ such that $E=F[\gamma]$.
    \label{primitive}
\end{theorem}

%%-------------------------

\section{Kato's Swan conductors and Applications to Brauer groups}

Let $F$ be a field with a discrete valuation $v$. Let $\mathcal{O}_F$ be the valuation ring 
\[
\mathcal{O}_F = \{x\in F:\ v(x)\geq 0\},
\]
and let $k\in \mathcal{O}_F/\mathcal{m}$ be the residue field.

When the discretely valued field $F$ is henselian of characteristic $p>0$, Kato \cite{MR991978}*{Page 110} and Izhboldin \cite{MR1386649} analyzed the {\it wild quotient} of $H^{n+1,n}(F)=H^{n+1}(F, \mathbb{Z}/p(n)) = H^1(F,\Omega^n_{X,\log})$ (Section \ref{etalemotivic}), which is the quotient of this group by its tamely ramified part (defined below). Totaro \cite{MR4411477} generalizes the result to arbitrary discrete valuation fields.

Kato defined an increasing filtration of $H^{n+1,n}(F)$ as follows: For $i\geq 0$, let $M_i$ be the subgroup of $H^{n+1,n}(F)$ generated by elements of the form 
\[
a\dfrac{db_1}{b_1}\wedge\cdots\wedge\frac{db_n}{b_n}
\]
with $a \in F$, $b_1,\dots, b_n \in F^\times$, and $v(a) \geq -i$. It is clear that
\[
0\subset M_0\subset M_1\subset \cdots,
\]
with $\bigcup_{i\geq 0}M_i=H^{n+1,n}(F)$. 

Let $t\in \mathcal{O}_F$ be a uniformizer for $v$. For any $j>0$, we define two homomorphisms depending on whether $j$ is relatively prime to $p$ or $p \mid j$. In each case, a simple computation shows that the homomorphim is well defined up to a choice of a uniformizer. First, consider the case when $j$ is relatively prime to $p$. We define
\[
\Omega^n_k\rightarrow M_j/M_{j-1}
\]
by
\[
\bar{a}\frac{d\bar{b}_1}{\bar{b}_1}\wedge\cdots\wedge\frac{d\bar{b}_n}{\bar{b}_n}\mapsto \frac{a}{t^j}\frac{db_1}{b_1}\wedge\cdots\wedge\frac{db_n}{b_n}\ (\textup{mod}\ M_{j-1}),
\]
for $a\in \mathcal{O}_F$ and $b_1,\dots,b_n\in \mathcal{O}^\times_F$. 

Now we define the second homomorphism. Let $Z_k^n$ be the subgroup of closed forms in $\Omega^n_k$. For $j>0$ and $p\mid j$, define a homomorphism 
\[
\Omega^n_k/Z^n_k\oplus \Omega^{n-1}_k/Z^{n-1}_k\rightarrow M_j/M_{j-1}
\]
as follows: On the first summand, it is defined as 
\[
\bar{a}\frac{d\bar{b}_1}{\bar{b}_1}\wedge\cdots\wedge\frac{d\bar{b}_n}{\bar{b}_n}\mapsto\frac{a}{t^j}\frac{db_1}{b_1}\wedge\cdots\wedge\frac{db_n}{b_n}\ (\textup{mod}\ M_{j-1}),
\]
and for the second summand it is defined as 
\[
\bar{a}\frac{d\bar{b}_1}{\bar{b}_1}\wedge\cdots\wedge\frac{d\bar{b}_{n-1}}{\bar{b}_{n-1}}\mapsto\frac{a}{t^j}\frac{dt}{t}\wedge\frac{db_1}{b_1}\wedge\cdots\wedge\frac{db_{n-1}}{b_{n-1}}\ (\textup{mod}\ M_{j-1}),
\]
where $a\in \mathcal{O}_F$ and $b_1,\dots,b_n\in \mathcal{O}_F^\times$.

The homomorphisms are well defined (although they depend on the choice of uniformizer $t$). We recall Cartier's theorem in this context. It says that, for $k$ of characteristic $p>0$, the subgroups $Z^n_k$ of closed forms in $\Omega^n_k$ is generated by the exact forms together with the forms of the form $a^p(db_1/b_1)\wedge\cdots\wedge(db_n/b_n)$ \cite{MR1386649}*{Lemma 1.5.1}.

To describe the subgroup $M_0$, we need to describe tame extensions of $F$ \cite{MR4411477}. We fix a discrete valuation $v$ as above. An extension field of $F$ is called {\it tame} with respect to $v$  if it is a union of finite extensions of $F$ for which the extension of residue fields is separable and the ramification degree is invertible in the residue field $k$. Let $F_{\textup{tame}}$ be the maximal tamely ramified extension of $F$ (with respect to $v$) in a separable closure of $F$. Define the {\it tame} (or {\it tamely ramified}) subgroup of $H^i(F,\mathbb{Z}/p(j))$ by
\[
H^i_{\textup{tame}}(F,\mathbb{Z}/p(j)) = \textup{ker}\Bigl(
H^i(F, \mathbb{Z}/p(j))\rightarrow H^i(F_{\textup{tame}},  \mathbb{Z}/p(j))
\Bigr).
\]
There is residue homomorphism on the tamely ramified subgroup
\[
\partial_v:H^i_{\textup{tame}}(F,\mathbb{Z}/p(j))\rightarrow H^{i-1}(k,\mathbb{Z}/p(j-1)),
\]
characterized by the property that
\[
\partial_v(a\frac{dt}{t}\wedge\frac{db_1}{b_1}\wedge\cdots\wedge\frac{db_{n-1}}{b_{n-1}})=\bar{a}\frac{d\bar{b}_1}{\bar{b}_1}\wedge\cdots\wedge\frac{d\bar{b}_{n-1}}{\bar{b}_{n-1}},
\]
where $a\in\mathcal{O}_F, b_1,\dots,b_{n-1}\in \mathcal{O}_F^\times$. Note that this description of elements of the tamely ramified subgroup follows from the theorem below. Then we define the {\it unramified} subgroup $H^n_{\textup{nr}}(F,\mathbb{Z}/p(j))$ to be the kernel of the residue homomorphism $\partial_v$.
\begin{theorem}
[{\cite{MR4411477}*{Theorem 4.3}}]
\label{totarogeneral}
    Let $F$ be a field of characteristic $p>0$ with a discrete valuation $v$ and residue field $k$. Then $H^{n+1,n}(F)$ is the union of the increasing sequence of subgroups $M_0\subset M_1\subset \cdots$ defined above, with isomorphisms (depending on a choice of uniformizer in $F$):
    \[
     M_j/M_{j-1}\cong \left\{
 \begin{aligned}
&\Omega^n_k &\textup{if}\ j>0 \ &\textup{and}\ p\nmid j,\\
&\Omega^n_k/Z^n_k\oplus \Omega^{n-1}_k/Z^{n-1}_k&\textup{if}\ j>0 \ &\textup{and}\ p\vert j.
 \end{aligned}
 \right.
    \]
Moreover, $M_0$ is the tame subgroup $H^{n+1,n}_{\textup{tame}}(F) = H^{n + 1}_{\textup{tame}}(F,\mathbb{Z}/p(n))$ defined above, and there is a well-defined residue homomorphism on $M_0$, yielding an exact sequence
\[
\xymatrix{
0\ar[r]&  H^{n+1,n}_{\textup{nr}}(F)\ar[r]& H^{n+1,n}_{\textup{tame}}(F)\ar[r]_-{\partial_v}& H^{n,n-1}(k)\ar[r]& 0, 
}
\]
where $H^{n+1,n}_{\textup{nr}}(F)$ is the unramified subgroup with respect to $v$. Finally, if the field $F$ is henselian with respect to $v$, then $H^{n+1,n}_{\textup{nr}}(F)\simeq H^{n+1,n}(k).$
\end{theorem}
\begin{definition}[Kato's Swan conductor]
\label{defsw}
Let $F$ be a field of characteristic $p>0$ with a discrete valuation $v$ and residue field $k$ and $\omega\in H^{n+1,n}(F)$. We define (Kato's) Swan conductor $\text{sw}_{F,v}(\omega)$ as follows:
\[\text{sw}_{F,v}(\omega):=\min\{n\in \mathbb{Z}: \omega\in M_n\}.
\]
\end{definition}
We write $\text{sw}(\bullet)$ instead of $\text{sw}_{K,v}(\bullet)$ for the sake of convenience when the notation is unambiguous. Here $H^{n+1,n}(F)$ is the union of the increasing sequence of subgroups $M_0\subset M_1\subset \cdots$ as defined in Theorem \ref{totarogeneral}.
\begin{remark}
    We should point out that the above filtration $\{M_i\}_i$ of $H^{n+1,n}(F)=H^{n+1}(F,\mathbb{Z}/p(n))$ for henselian fields $F$ has been discussed in detail in  \cite{MR991978}*{Page 110}. The result there works for more general coefficients, namely $\mathbb{Z}/p^m$. We  hope for a similar result for coefficients $\mathbb{Z}/p^m$ for general discretely valued fields.
\end{remark}

%-------------------
\iffalse
In this section, $K$ denotes a complete discretely valued field with valuation ring $\mathcal{O}_K$, residue field $F$ and maximal ideal $m_K=(\pi)$. The valuation of $K$ is denoted by $v_K$. We assume that $\text{char}(K)=\text{char}(F)=p>0$ and $[F:F^p]=p$. We have $K=F((\pi))$ by Cohen's structure theorem for the equal characteristic case.

We will only define the Swan conductor for the $p$-torsion Brauer group of $K$, since the higher torsion part can be recovered from the $p$-torsion part. The advantage of dealing with $p$-torsion Brauer group is that we don't need to deal with the Witt vectors and the cohomology of vanishing cycles. The interested reader should find the original statements in the Chapter 3, \cite{MR991978}.

Recall that 
\begin{align}
\label{pbrauer}
    \tx{Br}_{p}(K)\simeq_{\delta_1} \Omega^1_K/((\phi-I)\Omega^1_K+dK).
\end{align}

For an element $\lambda$ of $\Omega_K^1$, we denote the class of $\lambda$ in $\Omega^1_K/((\phi-I)\Omega^1_K+dK)$ by $[\lambda]$. 

\begin{definition}
For each $n\geq 0,$ let $M_n$ be the subgroup of $\Omega^1_K/((\phi-I)\Omega^1_K+dK)$ generated by all elements of the form $[f\dlog{g}]$ such that $v_K(f)\geq -n$. Then, $M_0\subset M_1\subset M_2\subset \cdots$, and $\bigcup\limits_n M_n=\Omega^1_K/((\phi-I)\Omega^1_K+dK).$ Put $M_{-1}=0$. We define
\[\text{sw}_{K,v_K}([\lambda]):=\min\{n\in \mathbb{Z}: [\lambda]\in M_n\}.
\]
\end{definition}
\fi
%------------------

In this paper, we are mainly interested in the $p^\infty$-torsion Brauer group of the field $F$. Since the higher torsion part of Brauer groups can be recovered from the $p$-torsion part (Theorem \ref{hightop}), we focus on the $p$-torsion Brauer group $\text{Br}(F)[p]\cong H^{2,1}(F)$ in the rest of this section. Then Theorem \ref{totarogeneral} gives: 

\begin{theorem}[Structure Theorem for $p$-torsion part of Brauer groups]
\label{sbr}
    Let $F$ be a field of characteristic $p>0$ with a discrete valuation $v$ and residue field $k$. Then $\text{Br}(F)[p]$ is the union of the increasing sequence of subgroups $M_0\subset M_1\subset \cdots$ defined above, with isomorphisms (depending on a choice of uniformizer in $F$):
    \[
     M_j/M_{j-1}\cong \left\{
 \begin{aligned}
&\Omega^1_k &\textup{if}\ j>0 \ &\textup{and}\ p\nmid j,\\
&\Omega^1_k/Z^1_k\oplus k/k^p&\textup{if}\ j>0 \ &\textup{and}\ p\vert j.
 \end{aligned}
 \right.
    \]

Moreover, $M_0$ is the tame subgroup $\text{Br}_{\textup{tame}}(F)[p]$, and there is a well-defined residue homomorphism on $M_0$ yielding an exact sequence
\[
\xymatrix{
0\ar[r]&  \text{Br}_{\textup{nr}}(F)[p]\ar[r]& \text{Br}_{\textup{tame}}(F)[p]\ar[r]_-{\partial_v}& H^1(k,\mathbb{Z}/p)\ar[r]& 0, 
}
\]
where $\text{Br}_{\textup{nr}}(F)[p]$ is the unramified subgroup with respect to $v$. Finally, if the field $F$ is henselian with respect to $v$, then $\text{Br}_{\textup{nr}}(F)[p]\cong \text{Br}(k)[p].$
\end{theorem}

The following lemma explains how the Swan conductor of the class $a\dlog{1+b}$ is affected by the valuations of $a, b$.
\begin{lemma}[Kato \cite{MR991978}]
Let $a, b\in F, \ i, j\in \mathbb{Z}$, and assume that $v_F(a)\geq -i,\ v_F(b)\geq j>0$. Then we have
\begin{align}
    a\dlog{1+b}\in M_{i-j}.
\end{align}
More precisely, if $a\neq 0$, we have
\label{swift}
\begin{align}
    a\dlog{1+b}+ab\dlog{a}\in M_{i-2j}.
\end{align}
\end{lemma}

\begin{proof}
\begin{align*}
    a\dlog{1+b}&=\frac{a}{1+b}d(1+b)\\
    &\equiv-(1+b)d(\frac{a}{1+b})\ \text{mod}\ d(F)\\
    &\equiv-bd(\frac{a}{1+b}) \ \text{mod}\ d(F)\\
    &= -(ab)d(\frac{1}{1+b})-(\frac{b}{1+b})da \\
    &\equiv-\frac{b}{1+b}da \ \text{mod}\ M_{i-2j}\\
     &\equiv-bda \ \text{mod}\ M_{i-2j}\\
    &=-(ab)\dlog{a} \ \text{mod}\ M_{i-2j}.
\end{align*}
\end{proof}

The next proposition is the key proposition for the period-index result for complete discretely valued field of characteristic $p>0$.
\begin{prop}[\cite{MR550688}*{Section 4, Lemma 5}]
\label{wildram}
Let $F$ be a complete field with a discrete valuation $v$ and residue field $k$. Suppose that $\textup{char}(k)=p>0$ and $[k:k^p]=p$. Suppose that $\omega \in \textup{Br}(F)[p]$ and $\omega \notin \text{Br}(F_{\text{tame}}/F)[p]$. Then the division algebra $D$ which represents $\omega$ is a degree $p$ division algebra whose residue algebra is a purely inseparable field extension of degree $p$ over $k$.\par
Moreover, suppose that $\text{char}(F)=p>0$. Let $\pi$ be a prime element of $F$. In this case, $D=[a,b)$ 
 where $a\in F, b\in F^\times$, and it must have one of the following two forms:
\begin{align*}
&\romannumeral 1)\ [\dfrac{f}{\pi^{pm}},e\pi), \tx{where}\ f\in \mathcal{O}_F,\ \Bar{f}\notin k^p,m>0,v(e)=0.
\\
&\romannumeral 2)\ [\dfrac{c}{\pi^n},g), \tx{where}\ g\in \mathcal{O}_F,\ \Bar{g}\notin k^p,v(c)=0 \tx{and}\ n \tx{is prime to $p$}.
\end{align*}
\end{prop}
In both case, $D$ is decomposed by a totally ramified field extension of degree $p$ and a field extension of degree $p$ whose residue field is a purely inseparable extension.

We notice that Kato's proof can be generalized to the henselian case easily. Hence, we put the generalized result below with proof: 

\begin{theorem}
\label{henselkato}
Let $F$ be a henselian field of characteristic $p$ with a discrete valuation $v$ and residue field $k$. Suppose that $[k:k^p]=p$. Suppose that $\omega \in \textup{Br}(F)[p]$ and $\omega \notin \text{Br}(F_{\text{tame}}/F)[p]$. Then the division algebra $D$ which represents $\omega$ is a degree $p$ divison algebra with inseparable residue field extension.  \par
Moreover, let $\pi$ be a prime element of $F$. Then $D=[a,b)$ for some $a\in F,b\in F^\times$ and it has one of the following two forms:
\begin{align*}
&\romannumeral 1)\ [\dfrac{f}{\pi^{pm}},e\pi), \tx{where}\ f\in \mathcal{O}_F,\ \Bar{f}\notin k^p,m>0,v(e)=0.
\\
&\romannumeral 2)\ [\dfrac{c}{\pi^n},g), \tx{where}\ g\in \mathcal{O}_F,\ \Bar{g}\notin k^p,v(c)=0 \tx{and}\ n \tx{is prime to $p$}.
\end{align*}
\end{theorem}
Before the proof, we need the following lemma utilizing the condition $[k:k^p]=p$.
\begin{lemma}
$\Omega^1_k/Z_k^1=0$
when $[k:k^p]=p$.
\end{lemma}
\begin{proof}
Since $[k:k^p]=p$, it follows that there is no nontrivial $2$-form over $k$. Therefore, every $1$-form is closed and $\Omega^1_k/Z_k^1=0$.
\begin{comment}
    Let $\{t\}$ be a $p$-basis of $k$. Then for all $x dt\in \Omega^1_k$, we have
\begin{align*}
    xdt &=(\sum_i x_i^p t^i) dt\\
     &=(\sum_{i<p-1} x_i^p d(\frac{t^{i+1}}{i+1}))+x_{p-1}^p t^{p-1}dt\\
     &=(\sum_{i<p-1} d(x_i^p \frac{t^{i+1}}{i+1}))+ x_{p-1}^p t^{p-1}dt \in Z_1\Omega^1_k.
\end{align*}
Hence, $\Omega^1_k/Z_1\Omega_k^1=0$.
\end{comment}
\end{proof}

So in the case when $p |n$ and $[k:k^p] = p,$ we have $M_n/M_{n-1}\cong k/k^p$, This follows from Theorem \ref{sbr}. We use the notation $U_F^{(n)}:=\{x\in F^\times\mid v_F(x-1)\geq n\}$ for each integer $n\geq 0$.

\begin{proof}[Proof of Theorem \ref{henselkato}]
$ $\\
The proof follows from the following two steps by induction on $i$:

(\romannumeral 1) Hypotheses: \begin{align*}& \omega\in \text{Br}(F)[p],\\
&\omega\equiv[\frac{f}{\pi^{pm}}\dlog{\pi}]\ \text{mod} \ M_i,\\
& f\in \mathcal{O}_F,\ \bar{f}\notin k^p,\ pm>i\geq 0,\\
&\pi \ \text{is a prime element of}\  F.
\end{align*}

Conclusion:
\begin{align*}
    &\text{There exist $f'$ and $\pi'$ such that}\\
    &\omega\equiv [\frac{f'}{\pi^{pm}}\dlog{\pi'}]\ \text{mod} \ M_{i-1}\\
    &v(f'-f)\geq pm-i,\pi'/\pi\in U_F^{(pm-i)}.
\end{align*}
(\romannumeral 2) Hypotheses: \begin{align*}& \omega\in \text{Br}(F)[p],\\
&\omega\equiv[\frac{c}{\pi^{n}}\dlog{g}]\ \text{mod} \ M_i,\\
&g\in \mathcal{O}_F,\ \Bar{g}\notin k^p,v(c)=0, n\geq i\geq 0\\
&\pi \ \text{is a prime element of}\  F.
\end{align*}

Conclusion:
\begin{align*}
    &\text{There exist $c'$ and $g'$ such that}\\
    &\omega\equiv [\frac{c'}{\pi^{n}}\dlog{g'}]\ \text{mod} \ M_{i-1}\\
    &v(c'-c)> n-i,g'/g\in U_F^{(n-i)}.
\end{align*}
We prove both of these simultaneously in two cases: $i > 0$ and $i = 0$.
\\
(1) $\bm{i>0}:$
For (\romannumeral 1), if $p|i$, the conclusion is clear since $M_i/M_{i-1}\cong k/k^p$ by fixing the uniformizer $\pi$. If $p\nmid i$, for any $g\in \mathcal{O}_F$, by Lemma \ref{swift},
\begin{align*}
    \frac{f}{\pi^{pm}}\dlog{1+h\pi^{pm-i}}&\equiv -\frac{fh\pi^{pm-i}}{\pi^{pm}}\dlog{\frac{f}{\pi^{pm}}}\\
    &= -\frac{fh}{\pi^i}\dlog{f} \ \text{mod}\ M_{pm-2(pm-i)}.
\end{align*}
Since $[k:k^p]=p$, we can find $h\in \mathcal{O}_F$ such that  $\omega-[\dfrac{f}{\pi^{pm}}\dlog{\pi(1+h\pi^{pm-i})}]\in M_{i-1}$. Hence the conclusion follows.

Then let us look at (\romannumeral 2). If $p\nmid i$, the conclusion follows since the $p$-rank of the residue field $k$ is $1$ and $M_i/M_{i-1}\simeq \Omega^1_k$ by fixing the uniformizer $\pi$. If $p\mid i$, then for any $e\in \mathcal{O}_F$, by Lemma \ref{swift} we have
\begin{align*}
    \frac{c}{\pi^{n}}\dlog{1+e\pi^{n-i}}&\equiv -\frac{ce\pi^{n-i}}{\pi^{n}}\dlog{\frac{c}{\pi^{n}}}\\
    &= -\underbrace{\frac{ce}{\pi^i}\dlog{c}}_{Sw(*)<i}-\frac{nce}{\pi^i}\dlog{\pi} \\
    &=-\frac{nce}{\pi^i}\dlog{\pi} \ \text{mod}\ M_{i-1}.
\end{align*}
The $(*)$ term has a Swan conductor smaller than $i$, since its residue corresponds to a term in $\Omega^1_k/Z_k^1\cong 0$. Now we can find $e\in \mathcal{O}_F$ such that $\omega-\dfrac{c}{\pi^n}\dlog{g(1+e\pi^{n-i})}\in M_{i-1}$. Hence the conclusion follows in this case.
\\
(2) $\bm{i=0}$:
First fix the uniformizer $\pi$. By Theorem \ref{sbr}, we have that \[
M_0\cong \text{Br}(k)[p]\oplus k/\mathcal{P}(k).
\]
For both hypotheses, the proof proceeds in two steps: First, we modify the condition in each hypothesis so that the resulting symbol algebra is congruent to $\omega$ modulo $\text{Br}(k)[p]$. In the second step, we finish the proof. We give details for $(\romannumeral 1)$ as an example.

Since $\omega-[\dfrac{f}{\pi^{pm}}\dlog{\pi}]\in M_0$, then there exists $f_1\in \mathcal{O}_F$ such that $\omega-[\dfrac{f+f_1\pi^{pm}}{\pi^{pm}}\dlog{\pi}]\in \text{Br}(k)[p]$. Set $f'=f+f_1\pi^{pm}$. For any $h\in \mathcal{O}_F$, by Lemma \ref{swift},
\begin{align*}
    \frac{f'}{\pi^{pm}}\dlog{1+h\pi^{pm}}&\equiv -f'h\dlog{\frac{f'}{\pi^{pm}}}\\
    &= -f'h\dlog{f'} \ \text{in}\ \text{Br}(F)[p].
\end{align*}
We can find $h\in \mathcal{O}_F$ such that $\omega-[\dfrac{f+f_1\pi^{pm}}{\pi^{pm}}\dlog{\pi(1+h\pi^{pm})}]\in \text{Br}(k)[p]$. Therefore the conclusion follows.

\end{proof}

Moreover, if $\text{Br}(k)=0$, we have the following:
\begin{theorem}
\label{sc1}
Let $F$ be a henselian field of characteristic $p$ with a discrete valuation $v$ and residue field $k$. Suppose that $[k:k^p]=p$ and $\text{Br}(k)[p]=0$.
For any $\omega\in \text{Br}(k)[p]$, the division algebra $D$ which represents $\omega$ is a symbol algebra of degree $p$.

Moreover, let $\pi$ be a prime element of $F$. The symbol algebra $D=[a,b),a\in F,b\in F^\times $ has one of the following three forms:
\label{structure}
\begin{align*}
&\romannumeral 1)\ [a,\pi), \tx{where}\ a\in \mathcal{O}_F,\ \Bar{a}\notin \mathcal{P}(k).
\\
&\romannumeral 2)\ [\dfrac{f}{\pi^{pm}},e\pi), \tx{where}\ f\in \mathcal{O}_F,\ \Bar{f}\notin k^p,m>0,v(e)=0.
\\
&\romannumeral 3)\ [\dfrac{c}{\pi^n},g), \tx{where}\ g\in \mathcal{O}_F,\ \bar{g}\notin k^p,v(c)=0 \tx{and}\ p\nmid n.
\end{align*}
\end{theorem}
\begin{proof}
This theorem is just a corollary of Theorem \ref{sbr} and Theorem \ref{henselkato}.
\end{proof}
\begin{remark}
\label{nrrf}
   In each case, the symbol algebra $D$ over $F$ is split by a totally ramified field extension of degree $p$ and a field extension of degree $p$ whose residue field is a purely inseparable extension or an Artin-Schreier extension. 
\end{remark}
The examples of residue fields satisfying the condition in Theorem \ref{sc1} are as follows: $k(C)$, $k((s))$, $k'((t))_{\textup{nr}}$, where $k$ is an algebraically closed field with characteristic $p>0$, $k'$ is a perfect field with characteristic $p>0$ and $k(C)$ denotes the function field of an algebraic curve $C$ over $k$.

\section{Period-index Problems of Henselian Discretely Valued fields}

Let $F$ be the henselian field with a discrete valuation $v$. Let $\mathcal{O}_F$ be the valuation ring $\{x\in F: v(x)\geq 0\}$, and let $k\in \mathcal{O}_F/\mathcal{m}$ be the residue field.

\begin{theorem}[ \cite{MR1626092}*{Proposition 2.1},  \cite{MR3413868}*{Proposition 6.1}]
Suppose that a field $F$ and all its finite extensions $E$, have the property that for all central simple $A/E$ of period $p$ satisfies $\text{ind}(A)\leq p^m$. Then, any $A/F$ of period $p^n$ satisfies $\text{ind}(A)\leq p^{mn}$. 
\label{takao}
\end{theorem}
The theorem says that the Brauer dimension at $p$ of a field $F$ satisfying the condition are controlled by the $p$-torsion Brauer classes over its finite extensions.

\begin{remark}
We apply Theorem \ref{takao} to classes of fields containing finite extensions. The well-known examples are the class of $C_m$ fields and the class of function fields of algebraic varieties of a given dimension.
\end{remark}
\begin{prop}
Suppose that $F$ is a henselian discretely valued field with the residue field $k$ of characteristic $p>0$ and $[k:k^p]=p^n$. Let $E$ be a finite extension of $F$. Then $E$ is also a henselian discretely valued field with the residue field $l$ and $[l:l^p]=p^n$.
\label{inse}
\end{prop}
\begin{proof}
We reduce to either case of a finite separable extension case or a purely inseparable simple extension case. When $E/F$ is finite separable, the statement follows from Lemma \ref{extofprank} and  \cite{stacks-project}*{\href{https://stacks.math.columbia.edu/tag/09E8}{Remark 09E8}}. When $E/F$ is a purely inseparable simple extension, the statement follows from \cite{stacks-project}*{\href{https://stacks.math.columbia.edu/tag/04GH}{Lemma 04GH}} and lemmas \ref{rlrins1}, \ref{rlrins2} below.
\end{proof}

\begin{lemma}[\cite{MR3219517}*{Lemma 3.1}]
\label{rlrins1}
    Let $B$ be a regular local ring with field of fractions $K$, residue field
$\kappa$ and maximal ideal $m$. Let $n$ be a natural number and $u\in B$ a unit such that
$[\kappa(u^{\frac{1}{n}}
) : \kappa] = n$. Then $B[u^{\frac{1}{n}}]$ is a regular local ring with residue field $\kappa(
\bar{u}^{\frac{1}{n}})$.
\end{lemma}

\begin{lemma}[\cite{MR3219517}*{Lemma 3.2}]
\label{rlrins2}
    Let $B$ be a regular local ring with field of fractions $K$, residue field
$\kappa$ and maximal ideal $m$. Let $\pi\in m$ be a regular prime and $n$ a natural number. Then $B[\pi^{\frac{1}{n}}]$ is a regular local ring with residue field $\kappa$.
\end{lemma}

Now we can give a proof for the first main result in the introduction. 
\begin{theorem}[Theorem \ref{firstthe}]
Let $F$ be a henselian discretely valued field of characteristic $p>0$ with the residue field $k$ such that $[k:k^p]=p$ and $\text{Br.dim}_p(l)=0$ for any finite extension $l/k$. Then we have $\text{Br.dim}_p(F)=1$.
\label{c1s}
\end{theorem}
\begin{proof}
$ $\newline
By Theorem \ref{sc1}, we have $\text{ind}(\omega)=\text{per}(\omega)$ for $\omega \in \text{Br}(F)[p]$. By Theorem \ref{takao} and Proposition \ref{inse}, we have $\text{Br.dim}_p  (F)=1$. 
\end{proof}

The following theorem was first proved in \cite{MR4453883}. We give a different proof using ideas developed in this article. 
\begin{theorem}[\cite{MR4453883}*{Theorem 2.3}]
\label{Chipchakov1}
 Let $F$ be a henselian discretely valued field of characteristic $p>0$ with the residue field $k$. Suppose that $k$ is a local field. Then $\text{Br.dim} (F)=1$.
\end{theorem}

\begin{proof}
Since $k$ is a local field of characteristic $p>0$, we have that $k\cong \mathbb{F}_q((s))$, $q=p^n$. By Theorem \ref{henselkato}, a wildly ramified Brauer class in $\text{Br}(F)[p]$ is represented by a symbol algebra of degree $p$. So it suffices to show that a tamely ramified Brauer class in $\text{Br}(F)$ has symbol length $1$. Let $\omega\in \text{Br} (F_{\text{tame}}/F)[p]$. Then $\omega=[a,\pi)+[b,c)$ where $a$ defines an unramified degree $p$ Artin-Schreier extension of $F$ and $[b,c)\in \text{Br}(k)[p]$. By \cite{MR554237}*{Corollary 3, Page 194}, $[b,c)$ is split by the degree $p$ Artin-Schreier extension defined by $a$. Hence, $\alpha=[a,e)$ for some $e\in F^\times$. 

Notice that a finite extension of a local field is still a local field. Hence, combining with Theorem \ref{takao}, we get the desired conclusion.
\end{proof}

\begin{remark}
Here we give a different proof of Chipchakov's result using Kato's Swan conductor when $\textup{char}(F)=p>0$ and $k$ is a local field. More generally, if $\text{char}(F)=0$, we can also use Kato's Swan conductor to give a proof as Proposition \ref{wildram} also works in the mixed characteristic case.
\end{remark}

\section{Ramification of a central division algebra over the field \texorpdfstring{$k((\pi))$}{Lg}}
In this section, let $F$ be a field of characteristic $p$ complete with respect to a discrete valuation $v$ and residue field $k$, where $[k:k^p]=p$. By the Cohen structure theorem, we have that $F\cong k((\pi))$, where $\pi$ is a uniformizer of $F$. We want to use the structure of $p$-torsion part of $\text{Br}\ k((\pi))$ to understand the ramification behavior of $p^n$-torsion part of $\textup{Br}\ k((\pi))$. We consider two cases: (1) $\text{Br.dim}_p (l)=0$ for all finite extension $l/k$, and (2) $k$ is a local field.
\subsection{\texorpdfstring{$\text{Br.dim}_p (l)=0$}{Lg} for all finite extension \texorpdfstring{$l/k$}{Lg}}
$ $\\
 In this case, $\text{Br}\ k((\pi))[p]$ has symbol length $1$. Hence, every central division algebra of {period $p$} over $k((\pi))$ is cyclic and has ramification index $p$ and a degree $p$ residue field extension. Now for a period $p^n$ central division algebra $A$ over $k((\pi))$, it has degree $p^n$ by Theorem \ref{c1s}. Hence we can assume $[A:k((\pi))]=p^{2n}$. This gives $e=e'=p^n$. The residue division algebra of $A$ is commutative and hence a field. It has degree $p^n$ over $k((\pi))$. We want to describe this residue field using the structure of $\text{Br}( k((\pi)))[p]$.

Next we explain how to read the information related to the residue field extension. Consider $p^{n-1}[A]$, where $[A]$ indicates the class of $A$ as above in $\text{Br}\ k((\pi))$. This class has period $p$. Hence, there exists a field extension $E_1/k((\pi))$ of degree $p$ with the degree $p$ residue field extension $l_1/k$ by Remark \ref{nrrf}. 

Then we consider $[A]_{E_1}$, the image of $[A]$ in $\text{Br}(E_1)$. The field $E_1$ is a complete discretely valued field with the residue field $l_1$ which is a degree $p$ field extension of $k$. Now $l_1/k$ is either an Artin-Schreier extension or a purely inseparable extension. By Proposition \ref{inse} and the assumptions on $k$, $\text{Br}(l_1)[p]=0$ and $[l_1:l_1^p]=p$. The period of $[A]_{E_1}$ is $p^{n-1}$. Otherwise, the index of $[A]_{E_1}$ is less than $p^{n-1}$. Then a splitting field of $[A]_{E_1}$ would have the degree over $k((\pi))$  less than $p^n(=p\cdot p^{n-1})$, which is a contradiction.

Hence, we can repeat the argument above to get a composition of field extensions of degree $p$, $k((\pi))\subset E_1\subset\cdots\subset E_n$, such that the composition of residue field extensions, $k\subset l_1\subset\cdots\subset l_n$, consists of either Artin-Schreier extension or purely inseparable extension of degree $p$. $E_n$ is a splitting field of $A$ with degree $p^n$.

This gives the following theorem:

\begin{theorem}
\label{pntorsion0}
Suppose that the field $k$ satisfies $[k:k^p]=p$ and $\text{Br.dim}_p(l)=0$ for all finite extensions $l/k$ .
Then the degree of a central division algebra of period $p^n$ over $k((\pi))$ is $p^n$. Moreover, it admits a splitting field of degree $p^n$ such that the residue field extension is of degree $p^n$ which is a composition of either Artin-Schreier extension of degree $p$ or purely inseparable extension of degree $p$. The ramification index is also $p^n$.
\end{theorem}
%\begin{proof}
%$ $\\
%It is contained in the discussion above.
%\end{proof}
In fact, we have an easy way to determine the separable degree and the inseparable degree of the residue field extension.
\begin{corol}
We continue with the same assumptions as in Theorem \ref{pntorsion0}. Let $A$ be a central division algebra over $k((\pi))$ of period $p^n$. Denote the order of $[A]_{\text{tame}}$ in $\text{Br}\ k((\pi))_\textup{tame} $ by $p^m(\leq p^n)$, where $k((\pi))_{\text{tame}}$ is the maximal tame extension of $k((\pi))$. Then the residue field $l$ of $A$ has degree $p^n$ over $k$ with separable degree $[l:k]_s=p^{n-m}$ and inseparable degree $[l:k]_i=p^m$.
\end{corol}
\begin{proof}
$ $\\
Consider the class $p^m[A]$. It is split by a tame extension of degree $p^{n-m}$ over $k((\pi))$. More precisely, this tame extension has ramification index $1$ and residual degree $p^{n-m}$. Hence, the proof reduces to the case $m=n$. This case just follows from Theorem \ref{pntorsion0}.
\end{proof}
\begin{comment}
For $\omega\in \text{Br}_{p^n}\ k((\pi))$, we can write it as $\omega=\omega_1+\cdots+\omega_n,$ where {\small $\omega_i=[a_i]_s\ \dlog{[b_i]_s}\in \text{Br}_{p^i}\ k((\pi))$}. It is easy to determine the ramification behavior of $a_i\ \dlog{b_i}\in \text{Br}_{p}\ k((\pi))$. 
Then it follows that
$\text{per}(\omega_{k((\pi))_{nr}})=\max\{~i\mid \ a_i\ \dlog{b_i} \ \text{wildly ramifies}\}.$
\end{comment}

\subsection{\texorpdfstring{$k$}{Lg} a local field}
$ $\\
In general, if $\text{Br.dim}_p(l)>0$ for a finite extension $l/k$, the situation is more complicated since there exist nontrivial division algebras over the residue field $k$. However, when $k$ is a local field, we have the following theorem similar to Theorem \ref{pntorsion0}.
\begin{theorem}
Suppose that $k$ is a local field, i.e.~$k\cong \mathbb{F}_q((t)),q=p^n.$
Then the degree of a central division algebra of period $p^n$ over $k((\pi))$ is $p^n$. Moreover, it admits a splitting field of degree $p^n$ such that the residue field extension is of degree $p^n$ and it is a composition of either Artin-Schreier extension of degree $p$ or purely inseparable extension of degree $p$. The ramification index is also $p^n$.
\end{theorem}
\begin{proof}
$ $\\
The proof is similar to the proof of Theorem \ref{pntorsion0}. The local field condition on $k$ is used for Theorem \ref{Chipchakov1}. It follows that a tamely ramified Brauer class in $\text{Br}\ k((\pi))[p]$ is split by a tame Artin-Schreier extension of degree $p$ over $k((\pi))$.
\end{proof}
Similarly, we have the following corollary.
\begin{corol}
Suppose that $k$ is a local field, i.e.~$k\cong \mathbb{F}_q((t)),q=p^n$. Let $A$ be a central division algebra over $k((\pi))$ of period $p^n$. Denote the order of $[A]_{\text{tame}}$ in $\text{Br}_{\text{tame}}\ k((\pi)) $ by $p^m(\leq p^n)$, where $k((\pi))_{\text{tame}}$ is the maximal tame extension of $k((\pi))$. Then the residue field $l$ of $A$ has degree $p^n$ over $k$ with separable degree $[l:k]_s=p^{n-m}$ and inseparable degree $[l:k]_i=p^m$.
\end{corol}
\begin{remark}
    In fact, the condition on $k$ can be replaced by $k$ is $p$-quasilocal and almost perfect using \cite{MR4453883}*{Theorem 2.3}. The key ingredient of the proof is  the fact that $\text{Br.dim}_p(k)=1$.
\end{remark}

 \section{
 Period-Index Problem for the Brauer group of an algebraic curve \texorpdfstring{$X$}{Lg} over \texorpdfstring{$k((t))$}
{Lg}
 }
We investigate the Brauer dimension at $p$ for the Brauer group of an algebraic curve $X$ over $k((t))$. Here we use the injection $\text{Br}(X)\hookrightarrow \text{Br}(K(X))$, where $F=K(X)$ is the function field of $X$ over a field $K=k((t))$ and $k$ is an algebraically closed field of characteristic $p>0$. Let $\mathcal{O}_K=k[[t]]$ be the complete discrete valuation ring with the field of fractions $K$ and $t$ the uniformizer. Denote by $T$ the unique closed point of $\text{Spec}(\mathcal{O}_K)$.    

\begin{definition}

    An integral model $\mathcal{X}$ of $X$ is a $2$-dimensional regular $\mathcal{O}_K$-scheme such that 
\begin{enumerate}
\item[(\romannumeral 1)] $p:\mathcal{X}\rightarrow \textup{Spec}(\mathcal{O}_K)$ is flat and proper;

\item[(\romannumeral 2)]
There is an isomorphism of $K$-schemes $X\cong \mathcal{X}_{K}$;

\item[(\romannumeral 3)] The reduced scheme $(Y=\mathcal{X}\times T)_{\text{red}}$ is a $1$-dimensional (proper) schemes over $T$ whose irreducible components are all regular and has normal crossings (i.e.~$\mathcal{X}_T$ only has ordinary double points as singularities).
\end{enumerate}
\end{definition}

The existence of an integral model follows from the resolution of singularities of excellent $2$-dimensional schemes (\cite{MR0491722}), and the embedded resolution of the special fiber (\cite{MR1917232}). If $X$ admits a smooth integral model $\mathcal{X}$ over $\mathcal{O}_K$, we say that $X$ has {\it good reduction} over $\mathcal{O}_K$. In this case, the special fiber $\mathcal{X}_T$ will have a single irreducible component that is  a proper smooth curve over $T$.

\[
\xymatrix{
 X\ar[r]\ar[d] & \mathcal{X} \ar[d] & Y \ar[l]\ar[d] \\
 \text{Spec}(K) \ar[r] & \text{Spec}(\mathcal{O}_K) & T=\text{Spec}(k) \ar[l]
}
\]

For a closed point $P$ of $\mathcal{X}$, let $\locr{\mathcal{X},P}$ denote the local ring at $P$, $\hat{\mathcal{O}}_{\mathcal{X},P}$ the completion of the regular local ring $\locr{\mathcal{X},P}$ at its maximal ideal and $F_P$ the field of fractions of $\hat{\mathcal{O}}_{\mathcal{X},P}$. For an open subset $U$ of an irreducible component of $Y$, let $R_U$ be the ring consisting of elements in $F$ which are regular on $U$. Then $\mathcal{O}_K\subset R_U$. Let $\hat{R}_U$ be the $(t)$-adic completion of $R_U$ and $F_U$ the field of fractions of $\hat{R}_U$.

Now suppose the algebraic curve $X$ has good reduction over $\mathcal{O}_K$. We have the following exact sequence by purity of the Brauer groups in codimension $1$:
\[
\xymatrix{
    0\ar[r] & \text{Br}(X) \ar[r] & \text{Br}(K(X)) \ar[r]^-{\oplus i_x} & \bigoplus\limits_{x\in X_0} \text{Br}(\text{Quot}(\hat{\locr{}}_{X,x})).
    }
\]

 For $\omega\in \text{Br}(X)$, we have that $\omega\in \text{Br}(\locr{X,x})$ for all $x\in X_0$. 

\begin{lemma}
Let $f:X\rightarrow Y$ be a morphism of schemes. Let $y\in Y$ and $q:X'=X\times_Y \text{Spec}\mathcal{O}_{Y,y}\rightarrow X$
be the projection morphism. Then $\mathcal{O}_{X',q^{-1}(x)}\simeq \mathcal{O}_{X,x}$ for any $x\in X_y$.
 \end{lemma}
 
 Recall that every effective irreducible divisor $D\subset \mathcal{X}$ is either $Y$ ($D$ is vertical), or the closure of a closed point $x\in X_0$ of the generic fiber ($D$ is horizontal). Using this lemma, it follows that $\omega\in \text{Br}(\mathcal{O}_{\mathcal{X},x})$ for all $x\in X_0\subset \mathcal{X}^1$. Hence we have that $\omega$ is ramified only along the vertical divisor $Y$. Hence we define Kato's Swan conductor of $\omega\in \text{Br}(X)$ in the following way.
 
 \begin{definition}[Swan conductor for Brauer groups of curves]
 \label{globalsw}
 Let $k$ be an algebraically closed field of characteristic $p>0$ and let $X$ be an algebraic curve over $k((t))$. Suppose $X$ has good reduction with the associated model $\mathcal{X}\rightarrow \textup{Spec} \ k[[t]]$. Denote by $v_Y$ the valuation associated to the divisor $Y$ and $F$ the function field $k(X)$.
 
 Then we define the $\mathcal{X}$-Swan conductor of $\omega\in \textup{Br}(X)[p]$ by
 \[
 \textup{sw}_\mathcal{X}(\omega)=\textup{sw}_{F,~v_Y}(\omega).
 \]
 \end{definition}

 Next we state the main result in this section. 
  \begin{theorem}
Let $X$ be a smooth projective curve over $k((t))$ where $k$ is an algebraically closed fields of characteristic $p>0$. Suppose there is a model $\mathcal{X}$ over $k[[t]]$ with good reduction. Suppose that $\omega\in \textup{Br}(X)[p]$ satisfies $\textup{sw}_{\mathcal{X}}(\omega)< p$. Then $\textup{per}(\omega)=\textup{ind}(\omega)$.    
      \label{swp}
  \end{theorem}

In order to prove Theorem \ref{swp}, we will use the patching method from \cite{MR2545681}, which reduces the global period-index problem to two types of local period-index problems. We continue to use the notation from the beginning of the section.

Let $\eta$ be a generic point of an irreducible component of $Y$ and $F_\eta$ the completion of $F$ at the discrete valuation given by $\eta$. Let $D$ be a central simple algebra over $F$. By \cite{MR3432268}*{5.8}, there exists an irreducible open set $U_\eta$ of $Y$ containing $\eta$ such that $\textup{ind}(D\otimes_F F_{U_\eta}) = \textup{ind}(D\otimes_F F_\eta)$.
\begin{theorem}[Patching, \cite{MR2545681}*{Theorem 5.1}\cite{MR3219517}*{Page 228}]
\label{patching}
Let $D$ be a central simple algebra over $F$ of period $p$.
Let $S_0$ be a finite set of closed points of $\mathcal{X}$ containing all the points of intersection of the components of $Y$ and the support of the ramification divisor of $D$. Let $S$ be a finite set of closed points of $\mathcal{X}$ containing $S_0$ and $Y\setminus (\cup\ U_{\eta})$, where $\eta$ varies over generic points of $Y$. Then
\[
\textup{ind}(D)=\lcm \left\{\textup{ind}(D\otimes F_\zeta)\right\},
\]
where $\zeta$ runs over $S$ and irreducible components of $Y\setminus S$.
\end{theorem}

We apply this theorem in our situation. First, suppose $\zeta=U$ for some irreducible component $U$ of $Y\setminus S$. Let $\eta$ be the generic point of $U$. Then $U\subset U_\eta$. Since $F_{U_\eta} \subset F_U$, $\text{ind}(D\otimes_F F_U)\mid \text{ind}(D\otimes_F F_{U_\eta})= \textup{ind}(D\otimes_F F_\eta)$.  Since the residue field of the generic point of $U$ is a function field of the curve over an algebraically closed field, by Theorem \ref{sc1}, we have $\text{ind}(D\otimes F_\eta)\vert p$. Hence, $\text{ind}(D\otimes_F F_U)\mid p$.

Next suppose $\zeta=P\in S$, where $P$ is a closed point of $\mathcal{X}$.
By the Cohen structure theorem for an equi-characteristic field \cite{stacks-project}*{\href{https://stacks.math.columbia.edu/tag/0C0S}{Tag 0C0S}}, we have 
\[
\hat{\mathcal{O}}_{\mathcal{X},P}\simeq k[[\pi,t]],
\]
where $\pi,t$ are local uniformizers at $P$. Notice that it is actually a $k$-algebra isomorphism, since the residue field $k$ is naturally embedded into the complete local ring. In general, the Cohen's structure theorem only provides a ring isomorphism instead of a $k$-algebra isomorphism.

To analyze the period-index problem for the field $F_P= k((\pi,t))$, we start with a Gersten-type theorem on the logarithmic de Rham-Witt cohomology of $k[[\pi,t]]$, which is analogous to the Artin-Mumford ramification sequence in the case of $p$-torsion Brauer groups over fields of characteristic $p>0$.

\begin{theorem}[Gersten-type theorem \cite{MR2396000}]
\label{Gersten}
Let $X$ be  the spectrum of a $2$-equidimensional regular local ring over $\mathbb{F}_p$ with the unique closed point $P$ and quotient field $K$. Then we have an exact sequence
\[
\xymatrix{
0\ar[r]& H^1(X,\Omega^1_{X,\text{log}})\ar[r]& H^1(K,\Omega^1_{K,\text{log}})\ar[r]^-{\delta_1}& \bigoplus\limits_{x\in X^1} H^2_x(X,\Omega^1_{X,\text{log}})\ar[r]^-{\delta_2}&  H^3_P(X,\Omega^1_{X,\text{log}})
\ar[r] & 0.
}
\]
\end{theorem}
The morphisms in the Gersten-type sequence are induced by the connecting morphisms in the \'etale local cohomology. We are going to interpret the morphisms $\delta_1$ and $\delta_2$ concretely.
\subsection{The Morphism \texorpdfstring{$\delta_1$}{Lg}} 
$ $\\
Let $x\in X^1$. By the \'etale excision theorem \cite{milne2022}, we have the following lemma.
\begin{lemma}
\[ H^2_x(X,\Omega^1_{X,\textup{log}})=  H^2_x(X_x,\Omega^1_{\bullet,\textup{log}})\cong H^2_x(\mathcal{O}^h_{X,x},\Omega^1_{\bullet,\textup{log}}).
\]
\end{lemma}
\begin{proof}
$ $\\
    The first isomorphism follows from the \'etale excision theorem. The second isomorphism follows from \cite{MR0559531}*{Corollary 1.28}. 
\end{proof}

Furthermore, we have the following commutative diagram:
\[
\xymatrix{
0\ar[r]& H^1(\mathcal{O}^h_{X,x},\Omega^1_{\bullet,\textup{log}}) \ar[r]\bijar[d] & H^1(K^h,\Omega^1_{\bullet,\textup{log}}) \ar[r]^-{\delta_1} \bijar[d] & H^2_x(\mathcal{O}^h_{X,x},\Omega^1_{\bullet,\textup{log}}) \ar[r]\ar@{.>}[d] &0\\
0\ar[r]& \text{Br}(\mathcal{O}^h_{X,x})[p]\ar[r] & \text{Br}(K^h)[p] \ar[r]^-{\delta_1'}  & \text{Br}(K^h)[p]/\text{Br}(\mathcal{O}^h_{X,x})[p] \ar[r] &0
}
\]
The first exact row comes from the long exact sequence in local cohomology in \'etale topology and the second exact row is the canonical exact sequence.

It follows that
\begin{equation}
\label{identdelta1}
H^2_x(X,\Omega^1_{X,\textup{log}})\cong \text{Br}(K^h)[p]\left/\text{Br}(\mathcal{O}^h_{X,x})[p]\right.,
\end{equation}
and we can identify the morphism $\delta_1$ with $\delta_1'$, i.e. $\delta_1=\delta_1'.$

\subsection{The Morphism \texorpdfstring{$\delta_2$}{Lg}} 
Let $y\in X^1$ and $Y$ be the closure of $y$ in $X$.
 \begin{lemma}[\cite{MR2396000}]
    Let $X,Z$ be regular schemes over $\mathbb{F}_p$ and let $i: Z\hookrightarrow X$ be a regular closed immersion of codimension $r$. Then we have $\underline{H}^j_Z(X,\mathcal{O}_X)=0,\underline{H}^j_Z(X,W_m\Omega^1_X)=0,\underline{H}^j_Z(X,W_m\Omega^i_X/dV^{m-1}\Omega^{i-1}_X)=0$ for $j\neq r$.
\end{lemma}
\begin{corol}
    Let the notation be as above. Then we have $\underline{H}^j_Z(X,W_m\Omega^i_{X,\textup{log}})=0$ for $j\neq r,r+1$.
\end{corol}
Also we have the following exact diagram:
\begin{align}
\xymatrix{
&0\ar
[d]& 0\ar[d] & 0\ar[d]\\
0\ar[r]&{H}^1_{Y}(X,\Omega^1_{\bullet,\textup{log}})\ar[r]\ar[d]& {H}^1_{y}(X-\{P\},\Omega^1_{\bullet,\textup{log}})\ar[r]^-{\delta^y_{1}}\ar[d]& {H}^2_{P}(X,\Omega^1_{\bullet,\textup{log}})\ar[d]\ar[r] & 0\\
0\ar[r]&{H}^1_{Y}(X,\Omega^1_\bullet)\ar[r]\ar[d]^-{F-I}& {H}^1_{y}(X-\{P\},\Omega^1_\bullet)\ar[r]^-{\delta^y_{1}}\ar[d]^-{F-I}& {H}^2_{P}(X,\Omega^1_\bullet)\ar[d]^-{F-I}\ar[r] & 0\\
0\ar[r]&{H}^1_{Y}(X,\Omega^1_\bullet/ d\mathcal{O})\ar[r]\ar[d]& {H}^1_{y}(X-\{P\},\Omega^1_\bullet/ d\mathcal{O})\ar[r]^-{\delta^y_{1}}\ar[d]& {H}^2_{P}(X,\Omega^1_\bullet/ d\mathcal{O})\ar[d]\ar[r] & 0\\
{H}^2_{P}(X,\Omega^1_{\bullet,\textup{log}})\ar[r]&{H}^2_{Y}(X,\Omega^1_{\bullet,\textup{log}})\ar[r]\ar[d]& {H}^2_{y}(X-\{P\},\Omega^1_{\bullet,\textup{log}})\ar[r]^-{\delta^y_{2}}\ar[d]& {H}^3_{P}(X,\Omega^1_{\bullet,\textup{log}})\ar[d]\ar[r] & 0\\
&0& 0 & 0
}
\label{comm1}
\end{align}
Notice that we have 
${H}^j_y(X,\bullet)={H}^j_y(X-\{P\},\bullet)$ by excision. In the above diagram, we are using cohomology groups instead of cohomology sheaves, since $X$ is a strictly henselian local scheme.

In order to compute $\delta_2$ and $H^3_P(X,\Omega^1_{X,\textup{log}})$, recall the following facts about the (\'etale) local cohomology.
\begin{lemma}[\cite{stacks-project}*{\href{https://stacks.math.columbia.edu/tag/0G74}{Lemma 0G74}}]
   Let $(X,\mathcal{O}_X)$ be a ringed space. Let $Z\subset X$ be a closed subset. Let $K$ be an object of $D(\mathcal{O}_X)$ and denote $K_{\text{ab}}$ its image in $D(\underline{\mathbb{Z}}_X)$. Then there is a canonical map $R\Gamma_Z(X,K)\rightarrow R\Gamma_Z(X,K_{\text{ab}})$ in $D(\text{Ab})$.
\end{lemma}
\begin{prop}[\cite{stacks-project}*{\href{https://stacks.math.columbia.edu/tag/0A46}{Lemma 0A46}}]
\label{comparlocal}
    Let $S$ be a scheme. Let $Z\subset S$ be a closed subscheme. Let $\mathcal{F}$ be a quasi-coherent $\mathcal{O}_S$-module and denote $\mathcal{F}^a$ the associated quasi-coherent sheaf on the small \'etale site of $S$. Then
    \[
    H^q_Z(S_{\text{Zar}},\mathcal{F})=H^q_Z(S,\mathcal{F}^a).
    \]
\end{prop}
\begin{prop}
\label{etalesheafomega}
    For any \'etale morphism $f:X\rightarrow Y$, $f^*\Omega^1_Y\rightarrow \Omega^1_X$ is an isomorphism of $\mathcal{O}_X$-modules.
\end{prop}

By Proposition \ref{etalesheafomega}, we get $(\Omega^1_S)^a=\Omega^1_S$ on the small \'etale site of $S$. It is also known that $(\mathcal{O}_S)^a=\mathcal{O}_S \ (\text{or} \ G_a)$, where $G_a$ is the additive group.
Then by Proposition \ref{comparlocal}, the \'etale local cohomology groups of $X$ agree with  the Zariski local cohomology groups.

Now it suffices to calculate the \'etale local cohomology groups of $\Omega^1_X/d\mathcal{O}_X$. In fact, we have the following exact sequences on the small \'etale site of $X$
\begin{align}
&\xymatrix{
0\ar[r]&\mathcal{O}_X \ar[r]^-F& \mathcal{O}_X \ar[r]& d\mathcal{O}_X\ar[r] & 0,
}\\
\label{nonc}
&\xymatrix{
0\ar[r]&d\mathcal{O}_X \ar[r]& \Omega^1_X \ar[r]& \Omega^1_X/d\mathcal{O}_X\ar[r] & 0.
}
\end{align}

These sequences follow from the Cartier isomorphism \cite{MR2396000}*{Corollary 2.5}, since $X$ is affine regular and $F$-finite. Passing to the cohomology sequence of \ref{nonc}, we have the exact sequences
\begin{align}
\xymatrix{
& 0\ar[r] & H^1_Y(X,\mathcal{O}_X)/H^1_Y(X,\mathcal{O}_X)^p \ar[r]^-d & H^1_Y(X,\Omega^1_X)\ar[r]& H^1_Y(X,\Omega^1_X/d\mathcal{O}_X)\ar[r]& 0.
\\
& 0\ar[r] & H^2_P(X,\mathcal{O}_X)/H^2_P(X,\mathcal{O}_X)^p \ar[r]^-d & H^2_P(X,\Omega^1_X)\ar[r]& H^2_P(X,\Omega^1_X/d\mathcal{O}_X)\ar[r]& 0.
}
\label{cartier}
\end{align}
Using (\ref{cartier}), it suffices to calculate the local cohomology groups of $\Omega^1_X$ and $\mathcal{O}_X$. They are computed by the Cech complex in the below.
\begin{lemma}[\cite{stacks-project}*{\href{https://stacks.math.columbia.edu/tag/0A6R}{Lemma 0A6R}}]
\label{Koszul}
    Let $A$ be a noetherian ring and let $I=(f_1,\cdots,f_r)\subset A$ be an ideal. Set $Z=V(I)\subset \text{Spec}(A)$. Then
    \[
    R\Gamma_Z(A)\simeq (\xymatrix{A\ar[r]& \prod_{i_0}A_{f_{i_0}}\ar[r]&\cdots\ar[r]& A_{f_1\cdots f_r}})
    \]
    in $D(A)$. If $M$ is an $A$-module, then we have
    \[
    R\Gamma_Z(M)\simeq (\xymatrix{M\ar[r]& \prod_{i_0}M_{f_{i_0}}\ar[r]&\cdots\ar[r]& M_{f_1\cdots f_r}})
    \]
    in $D(A)$.
\end{lemma}
Recall that $X=\text{Spec}\ k\ [[\pi,t]]$, where $k$ is an algebraically closed field. Let $R=k\ [[\pi,t]]$. Then $\pi$ and $t$ are regular primes of $R$. Denote by $V(\pi),V(t)$ the closures of codimension $1$ points $(\pi)$ and $(t)$ respectively.

Then we have the following 
\begin{align*}
{H}^1_{V(t)}(X,\Omega^1_X)&\cong \Omega^1_{R[\frac{1}{t}]}\left/\Omega^1_R\right.,
\\
{H}^1_{(t)}(X,\Omega^1_X)&\cong {H}^1_{(t)}(D(\pi),\Omega^1_{D(\pi)}) \cong {
\Omega^1_{R[\frac{1}{\pi t}]}\left/\Omega^1_{R[\frac{1}{\pi}]}\right.},\\
{H}^2_P(X,\Omega^1_X)&\cong {\Omega^1_{R[\frac{1}{\pi t}]}\left/(\Omega^1_{R[\frac{1}{\pi}]}+\Omega^1_{R[\frac{1}{t}]})\right.}\\
{H}^1_{V(t)}(X,\mathcal{O}_X)&\cong {R[\frac{1}{t}]}\left/{R}\right..\\
{H}^1_{(t)}(X,\mathcal{O}^1_X)&\cong {H}^1_{(t)}(D(\pi),\mathcal{O}^1_{D(\pi)}) \cong {
{R[\frac{1}{\pi t}]}\left/{R[\frac{1}{\pi}]}\right.},\\
{H}^2_P(X,\mathcal{O}^1_X)&\cong {{R[\frac{1}{\pi t}]}\left/\Bigl({R[\frac{1}{\pi}]}+{R[\frac{1}{t}]}\Bigr)\right.}.
\end{align*}
\begin{comment}
    {H}^1_{(\pi)}(X,\Omega^1_\bullet)&\simeq {H}^1_{\pi}(U,\Omega^1_\bullet)\simeq {H}^1_{(\pi)}(D(t),\Omega^1_\bullet) \simeq {
\Omega^1_{k[[\pi,t]][\frac{1}{\pi t}]}\left/\Omega^1_{k[[\pi,t]][\frac{1}{t}]}\right.},
\\
{H}^2_P(X,\Omega^1_\bullet)&\simeq {\Omega^1_{k[[\pi,t]][\frac{1}{\pi t}]}\left/(\Omega^1_{k[[\pi,t]][\frac{1}{\pi}]}+\Omega^1_{k[[\pi,t]][\frac{1}{t}]})\right.},
\\
\\
{H}^1_{(\pi)}(X,\mathcal{O}_\bullet)&\simeq {H}^1_{\pi}(U,\mathcal{O}_\bullet)\simeq {H}^1_{(\pi)}(D(t),\mathcal{O}_\bullet) \simeq {
{k[[\pi,t]][\frac{1}{\pi t}]}/{k[[\pi,t]][\frac{1}{t}]}},
\\
{H}^2_P(X,\mathcal{O}_\bullet)&\simeq {{k[[\pi,t]][\frac{1}{\pi t}]}\left/\left({k[[\pi,t]][\frac{1}{\pi}]}+{k[[\pi,t]][\frac{1}{t}]}\right)\right.}.
\end{comment}
Combining with the exact sequence (\ref{cartier}), it follows that
\begin{align*}
    H^2_{V(t)}(X,\Omega^1_{X,\text{log}})&\cong {\Omega^1_{R[\frac{1}{t}]}\left/ \Bigl(\Omega^1_{R}+(F-I)\Omega^1_{R[\frac{1}{t}]}+d(R[\frac{1}{t}])\Bigr)
    \right.}\\
    &\cong \text{Br}(R[\frac{1}{t}])[p]\left/\text{Br} (R)[p]\right.,\\
    H^2_{(t)}(X,\Omega^1_{X,\text{log}})&= {\Omega^1_{R[\frac{1}{\pi t}]}\left/ \Bigl(\Omega^1_{R[\frac{1}{\pi}]}+(F-I)\Omega^1_{R[\frac{1}{\pi t}]}+d(R[\frac{1}{\pi t}])\Bigr)
    \right.}\\
    &\cong \text{Br}( R[\frac{1}{\pi t}])[p]\left/\text{Br}(R[\frac{1}{\pi}])[p]\right.,\\
     H^3_{P}(X,\Omega^1_{X,\text{log}})&= {\Omega^1_{R[\frac{1}{\pi t}]}\left/ \Bigl(\Omega^1_{R[\frac{1}{\pi}]}+\Omega^1_{R[\frac{1}{t}]}+(F-I)\Omega^1_{R[\frac{1}{\pi t}]}+d(R[\frac{1}{\pi t}])\Bigr)
    \right.}\\
    &\cong \text{Br}( R[\frac{1}{\pi t}])[p]\left/\Bigl(\text{Br}( R[\frac{1}{\pi}])[p]+\text{Br}(R[\frac{1}{t}])[p]\Bigr)\right..
\end{align*}
Notice that we used the fact that the localization of a unique factorization domain (UFD) at a multiplicatively closed subset is also a UFD and the Picard group of a UFD is zero.
Furthermore, from the last row of (\ref{comm1}), we get the following identification
{\small
\begin{align}
\label{h3log}
\xymatrix{
\text{Br}(R[\frac{1}{t}])[p]\left/\text{Br} (R)[p]\right.\ar[r]& \text{Br}( R[\frac{1}{\pi t}])[p]\left/\text{Br}(R[\frac{1}{\pi}])[p]\right.\ar[r]^-{\delta_2^y}& \text{Br}( R[\frac{1}{\pi t}])[p]\left/\Bigl(\text{Br}( R[\frac{1}{\pi}])[p]+\text{Br}(R[\frac{1}{t}])[p]\Bigr)\right.\ar[r] & 0.
}
\end{align}
}
\begin{comment}
Combining the results above, we can identify $\delta_2$ with the following quotient morphism
\begin{align}
&\Omega^1_{k[[\pi,t]][\frac{1}{\pi t}]}\left/\left[\Omega^1_{k[[\pi,t]][\frac{1}{t}]}+(F-I)\Omega^1_{k[[\pi,t]][\frac{1}{\pi t}]}+d({k[[\pi,t]][\frac{1}{\pi t}]})\right]\right.\\
\twoheadrightarrow
&\Omega^1_{k[[\pi,t]][\frac{1}{\pi t}]}\left/\left[\Omega^1_{k[[\pi,t]][\frac{1}{\pi}]}+\Omega^1_{k[[\pi,t]][\frac{1}{t}]}+(F-I)\Omega^1_{k[[\pi,t]][\frac{1}{\pi t}]}+d({k[[\pi,t]][\frac{1}{\pi t}]})\right].\right.
\label{delta2}
\end{align}
Meanwhile,
\begin{align*}
{H}^2_{(\pi)}(X,\Omega^1_{\bullet,\textup{log}})=&\ \Omega^1_{k[[\pi,t]][\frac{1}{\pi t}]}\left/\left[\Omega^1_{k[[\pi,t]][\frac{1}{t}]}+(F-I)\Omega^1_{k[[\pi,t]][\frac{1}{\pi t}]}+d({k[[\pi,t]][\frac{1}{\pi t}]})\right],\right.\\
{H}^3_{P}(X,\Omega^1_{\bullet,\textup{log}})=&\ \Omega^1_{k[[\pi,t]][\frac{1}{\pi t}]}\left/\left[\Omega^1_{k[[\pi,t]][\frac{1}{\pi}]}+\Omega^1_{k[[\pi,t]][\frac{1}{t}]}+(F-I)\Omega^1_{k[[\pi,t]][\frac{1}{\pi t}]}+d({k[[\pi,t]][\frac{1}{\pi t}]})\right].\right.
\end{align*}
\end{comment}
\subsection{Local Swan Conductor}
Recall that we have the $k$-algebra isomorphism 
$\hat{\mathcal{O}}_{\mathcal{X},P}\cong k[[\pi,t]]$. We have the commutative diagram
\[
\xymatrix{
&\text{Br}(F)[p]\cong H^1(F,\Omega^1_{F,\textup{log}})\ar[r]\ar[rd]  &H^1({\mathcal{O}}_{\mathcal{X},P},\Omega^1_{\bullet,\textup{log}})\ar[d]\ar[r] & H^2_{(t)}\left({\mathcal{O}}_{\mathcal{X},P},\Omega^1_{\bullet,\textup{log}}\right) \ar[d]\\
& &H^1(\hat{\mathcal{O}}_{\mathcal{X},P},\Omega^1_{\bullet,\textup{log}})\ar[r]&H^2_{(t)}\left(\hat{\mathcal{O}}_{\mathcal{X},P},\Omega^1_{\bullet,\textup{log}}\right) 
}
\]
where the horizontal row is part of the long exact sequence in local \'etale cohomology associated to the sheaf $\Omega^1_{\bullet, \textup{log}}$. By (\ref{identdelta1}), we have the following isomorhisms
\begin{align}
\label{h2s1}
&H^2_{(t)}({\mathcal{O}}_{\mathcal{X},P},\Omega^1_{\bullet,\textup{log}})\cong \text{Br}(\text{Frac}(({\mathcal{O}}_{\mathcal{X},{P}})^h_{(t)}))[p]/\text{Br}( ({\mathcal{O}}_{\mathcal{X},{P}})^h_{(t)})[p],\\
\label{h2s2}
&H^2_{(t)}(\hat{\mathcal{O}}_{\mathcal{X},P},\Omega^1_{\bullet,\textup{log}})\cong \text{Br}( \text{Frac}((\hat{\mathcal{O}}_{\mathcal{X},P})^h_{(t)}))[p]/\text{Br}( (\hat{\mathcal{O}}_{\mathcal{X},P})^h_{(t)})[p].
\end{align}
Notice that, for a prime ideal $\mathfrak{p}$ of a ring $R$, we denote by $R_\mathfrak{p}^h$ the henselization of the localization $R_\mathfrak{p}$ of the ring $R$ at the prime ideal $\mathfrak{p}$.

Using the diagram above, we can give a result which relates the global Swan conductor to the local cohomology groups  as in (\ref{h2s1}) and (\ref{h2s2}).
\begin{prop}
\label{localest}
Let $X$ be an algebraic curve over $k((t))$ with a smooth integral model $\mathcal{X}\rightarrow \text{Spec}\ k[[t]]$ and $\omega\in \textup{Br}(X)[p]$. Then
\[
\textup{sw}_{\text{Frac}((\hat{\mathcal{O}}_{\mathcal{X},P})^h_{(t)})}(\omega)=\textup{sw}_{\text{Frac}(({\mathcal{O}}_{\mathcal{X},P})^h_{(t)})}(\omega)=\textup{sw}_\mathcal{X}(\omega).
\]
\end{prop}

\begin{proof}[Proof of Proposition \ref{localest}]
$ $\\
The second equality follows from Definition \ref{globalsw}. For the first one, by Lemma \ref{swchange}, we can take $K=\text{Frac}(({\mathcal{O}}_{\mathcal{X},P})^h_{(t)}))$ and $L=\text{Frac}((\hat{\mathcal{O}}_{\mathcal{X},P})^h_{(t)})$. Then it suffices to show that the residue field extension is separable, since $t$ is the uniformizer in both fields. The residue field extension is given by $k(\pi)\rightarrow k((\pi))$, which is a completion morphism. The separability is given by Lemma \ref{comsep}. \qedhere
\end{proof}
\begin{lemma}[\cite{MR991978}*{Lemma 6.2, Page 119}]
\label{swchange}
Let $K\subset L$ be two henselian discretely valued fields such that $\mathcal{O}_K\subset \mathcal{O}_L$ and $m_L=\mathcal{O}_L m_K$. Assume that the residue field of $L$ is separable over the residue field of $K$. Then, for any $\omega\in \text{Br}(K)[p]$, we have
\[
\textup{sw}_K(\omega)=\textup{sw}_L(\omega).
\]
\end{lemma}
\begin{lemma}
\label{comsep}
Let $F$ be a discretely valued field with $[F:F^p]=p$ and $\hat{F}$ be its completion. Then the completion morphism $F\rightarrow \hat{F}$ is separable.
\end{lemma}

\begin{proof}
$ $\\
We prove it by contradiction. Suppose that there exists an algebraic extension $E/F$ inside $\hat{F}$ which is not separable. Then we can decompose $E/F$ as a chain of field extensions $E/L/F$ where $L$ is separable over $F$ and $E/L$ is purely inseparable. Moreover, let $\pi$ be a uniformizer of $F$. Then we have that $\pi$ is still a uniformizer in $L$, and it gives a $p$-basis of $L/L^p$. Since $E/L$ is purely inseparable, there exists $a\in E\subset\hat{F}$ such that $a^p\in L\setminus L^p$. Let $b = a^p$. It follows that $b=\sum\limits^{p-1}_{i=0}f_i^p\pi^i$ in $L$ (also in $\hat{F}$) such that $f_i\neq 0$ for some $i > 0$. However, notice that $\pi$ is a uniformizer of $\hat{F}$. Therefore, it implies that $b$ is not a $p$-power in $\hat{F}$, which is a contradiction. Hence the conclusion follows.
\end{proof}

\iffalse
Notice that the completion morphism $\mathcal{O}_{\mathcal{X},P}\rightarrow \hat{\mathcal{O}}_{\mathcal{X},P}$ is faithfully flat. We have the induced morphism $\mathcal{O}_{\mathcal{X},(t)}\rightarrow (\hat{\mathcal{O}}_{\mathcal{X},P})_{(t)}$. Passing to the henselization (completion), we have morphisms

$\mathcal{O}_{\mathcal{X},(t)}^h\rightarrow (\hat{\mathcal{O}}_{\mathcal{X},P})_{(t)}^h$ ($k(Y)[[t]]\rightarrow k((\pi))[[t]]$).

Combing with the above lemma, we get the following.

\fi

\subsection{Period-Index Problems at Closed Points}
\begin{theorem}
 Let $R=k[[\pi,t]]$, $X=\textup{Spec}(R)$, $K=\textup{Frac}(R)$ and $\omega\in \textup{Br}(K)[p]$ which ramifies only along $(t)$ with $\textup{sw}_{K,{(t)}}(\omega)=m< p$. Then $\omega=[*,\pi)$.
 \label{localperind}
\end{theorem}

\begin{proof}
By Theorem \ref{Gersten}, we have the exact sequence 
\[
\xymatrix{
0\ar[r]& H^1(X,\Omega^1_{X,\text{log}})\ar[r]& H^1(K,\Omega^1_{K,\text{log}}) \ar[r]^-{\delta_1}&  \bigoplus\limits_{x\in X^1} H^2_x(X,\Omega^1_{X,\text{log}})\ar[r]^-{\delta_2}&  H^3_P(X,\Omega^1_{X,\text{log}})
\ar[r] & 0.}
\]
Since $X$ is affine and regular, $H^1(X_\text{\'et},\Omega^1_{X,\text{log}})\simeq\text{Br}(R)[p]\simeq\text{Br}(k)[p]=0$. So the last exact sequence reduces to
\[
\xymatrix{
0\ar[r]& H^1(K,\Omega^1_{K,\text{log}})\ar[r]^-{\delta_1}&  \bigoplus\limits_{x\in X^1} H^2_x(X,\Omega^1_{X,\text{log}})\ar[r]^-{\delta_2}&  H^3_P(X,\Omega^1_{X,\text{log}})
\ar[r] & 0.}
\]
Our goal is to find a symbol algebra that represents $\omega$ in the Brauer group of $K$. Since $\textup{sw}_{K,{(t)}}(\omega)=m< p$, by Theorem \ref{structure},
we have that
\[
\omega=[\frac{f_m}{t^m},\pi)+[\frac{f_{m-1}}{t^{m-1}},\pi)+\cdots+[f_0,t) \ \text{in} \ \text{Br}(K)[p] ,
\]
where $f_i\in \pi \cdot k[[\pi]]$ for all $i$. The choices of $f_i$ follow from the identification in (\ref{h3log}). Since that $\pi$ is a regular element, we have  a consistent way to lift the elements from $k((\pi))$.

Since $f_0\in k[[\pi]]$  and $k$ is algebraically closed, by Hensel's lemma, we have $[f_0,t)\simeq 0$. Hence, $\omega=\displaystyle [\frac{f_m}{t^m}+\cdots+\frac{f_1}{t},\pi)$.
\end{proof}

Now we are ready to finish the proof of Theorem \ref{swp} based on Theorem \ref{patching}.

\begin{proof}[Proof of Theorem \ref{swp}]
    $ $\\
    First, suppose $\zeta=U$ for some irreducible component $U$ of $Y\setminus S$. Let $\eta$ be the generic point of $U$. Then $U\subset U_\eta$. Since $F_{U_\eta} \subset F_U$, $\text{ind}(D\otimes_F F_U)\mid \text{ind}(D\otimes_F F_{U_\eta})= \textup{ind}(D\otimes_F F_\eta)$.  Since the residue field of the generic point of $U$ is a function field of the curve over an algebraically closed field, by Theorem \ref{sc1}, we have $\text{ind}(D\otimes F_\eta)\vert p$. Hence, $\text{ind}(D\otimes_F F_U)\mid p$.

Second, suppose $\zeta=P\in S$, where $P$ is a closed point of $\mathcal{X}$. Combining Proposition \ref{localest} and Theorem \ref{localperind}, we have $\text{ind}(D\otimes F_\zeta)\vert p$.

Finally, by Theorem \ref{patching}, we have that $\text{per}(\omega)=\text{ind}(\omega)$.
\end{proof}

\bibliographystyle{plain} % We choose the "plain" reference style

\end{document}